\documentclass{article}
\usepackage[margin=25mm]{geometry}
\usepackage{lmodern}
\usepackage[utf8]{inputenc}
\usepackage[russian,english]{babel} 

\usepackage{amsmath,amssymb,amsthm,amscd,mathtools}
\usepackage{microtype,enumitem}
\usepackage[hidelinks]{hyperref}
\setlist{itemsep=2pt,topsep=4pt}
\newtheorem{theorem}{Theorem}[section]
\newtheorem{proposition}[theorem]{Proposition}
\newtheorem{lemma}[theorem]{Lemma}
\newtheorem{corollary}[theorem]{Corollary}
\theoremstyle{definition}
\newtheorem{definition}[theorem]{Definition}
\newtheorem{assumption}[theorem]{Assumption}
\theoremstyle{remark}
\newtheorem{remark}[theorem]{Remark}
\newtheorem{example}[theorem]{Example}
\newcommand{\E}{\mathbb E}
\newcommand{\Pp}{\mathbb P}
\newcommand{\R}{\mathbb R}
\newcommand{\N}{\mathbb N}
\newcommand{\F}{\mathbb F}
\newcommand{\cP}{\mathcal P}
\newcommand{\cM}{\mathcal M^2_0}
\newcommand{\one}{\mathbf 1}
\newcommand{\Jt}{\widetilde J}
\newcommand{\HJ}{\mathcal L^2_{\Pp}(\mu)}
\newcommand{\HX}{\mathcal L^2_{\Pp}([X])}
\newcommand{\HXJ}{\mathcal L^2_{\Pp}([X],\mu)}
\newcommand{\cIJ}{\mathcal I^2_{\Pp}(\mu)}
\newcommand{\IXJ}{\mathcal I^2_{\Pp}([X],\mu)}
\newcommand{\nablap}{\nabla^{\Pp}_{X,J}}

\newcommand{\CE}{\mathsf M}
\newcommand{\dadd}{D^{\mathrm{add}}}
\newcommand{\norm}[1]{\lVert#1\rVert}
\DeclareMathOperator{\Dom}{Dom}
\DeclareMathOperator{\Ran}{Ran}
\DeclareMathOperator{\tr}{tr}
\DeclareMathOperator{\Var}{Var}
\title{Functional calculus and martingale representation formula for\\
integer-valued random measures}
\author{Pierre Blacque-Florentin and  Rama Cont}
\date{Revision: 2026}

\begin{document}
\maketitle
\begin{abstract}
We develop a non-anticipative calculus for functionals of a continuous
martingale and an integer-valued random measure. We first construct a non-anticipative pathwise operator, shown to  be the inverse of the stochastic integration for a
dense functional class of integrands. We then show that this operator has a unique closure which is an isometry and  defines a {\it stochastic derivative} acting on stochastic integrals. This stochastic derivative yields a constructive martingale representation formula  and an orthogonal decomposition. The construction allows for random compensators, singular
continuous time dependence, and degenerate continuous covariation.
On Wiener–Poisson space, the stochastic derivative is shown to be the predictable projection of the Malliavin derivative; we also obtain a  factorization based on local mark gradients.

We exhibit several applications of this framework to integer-valued random measures:
 sensitivity formulas and a convergent expansion for expectations under bounded predictable perturbations of the compensator, a change-of-measure formula for the derivative, a decomposition of the  relative entropy  in terms of the stochastic derivative, and a pathwise refinement of the Mecke formula.
\end{abstract}
\noindent\textbf{Keywords:} Martingale representation, functional It\^o
calculus, integer-valued random measures, stochastic derivative,
Malliavin calculus, Poisson space,  Mecke formula, entropy.
\medskip
\setcounter{tocdepth}{2}
\tableofcontents
\newpage

\section{Introduction}

Let $X$ be a continuous $\R^d$-valued local martingale and $J$ an
integer-valued random measure on $(0,T]\times E$ with predictable compensator
$\mu$. We assume throughout that, almost surely, $J$ has at most one jump at
each time and no jumps at predictable times. Every square-integrable martingale
$Y$ then has a unique decomposition
\begin{equation}\label{eq:intro-rep}
Y_t=Y_0+\int_0^t\phi_s\cdot\mathrm dX_s
+\int_{(0,t]\times E}\psi(s,z)\,(J-\mu)(\mathrm ds,\mathrm dz)+N_t ,
\end{equation}
in which $N$ is orthogonal to all stochastic integrals with respect to $X$ and
$J-\mu$. The residual $N$ vanishes for every $Y$ exactly when $(X,J-\mu)$ has
the predictable representation property \cite{KW67,Jacod75,JS03}. The existence
of $(\phi,\psi)$ is a Hilbert-space projection argument. Their
\emph{identification} is what is needed in hedging, stochastic control and
backward stochastic differential equations, and it is the subject of this
paper.

Explicit formulas for the integrands are known in some settings. For marked
point processes in their natural filtration, they are expressed through the
conditional laws of the successive jump times and marks
\cite{ChouMeyer75,Davis76,BVW75,Elliott76}; see also \cite{Cohen13}. For
functionals of a Markov process at a fixed time, they are expressed through its
transition semigroup \cite{JacodMeleardProtter2000}. On Wiener and Poisson
spaces, the Clark--Ocone formula expresses them as the predictable projection
of an anticipative Malliavin derivative
\cite{Clark70,Haussmann79,Ocone84,Lokka04,LP11}. These results provide explicit formulas in the setting of
point processes, L\'evy processes and Markov processes.
The Markovian
formulas of Jacod, M\'el\'eard and Protter \cite{JacodMeleardProtter2000} also allow an
orthogonal martingale residual, while Last and Penrose \cite{LP11} treat
square-integrable Poisson functionals. 
Our construction does not require Markov structure, Malliavin differentiability or the representation property and
identifies the integrands through the closure of a pathwise
operator, allowing random, history-dependent compensators.

We identify the integrands $(\phi,\psi)$ as non-anticipative {\it stochastic derivatives} of $Y$. For a
non-anticipative functional $F$ of the stopped paths of $X$, $J$, $\mu$ and
$[X]$, the \emph{jump derivative} (Definition~\ref{def:jump}) is the effect of
inserting a mark at the current time:
\begin{equation}\label{eq:intro-jump}
\nabla_J F(t,z)=F\bigl(t,X^{t-},J^{t-}+\delta_{(t,z)},\mu^{t},[X]^{t}\bigr)
-F\bigl(t,X^{t-},J^{t-},\mu^{t},[X]^{t}\bigr),
\end{equation}
where $J^{t-}$ is the restriction of $J$ to $(0,t)\times E$. The continuous
component uses the vertical derivative $\nabla_X F$ of the functional It\^o
calculus \cite{CF09,CF10,Dupire09}. Both operators are defined path by path, without
reference to a probability measure, and on compensated integral functionals
they return the integrands (Proposition~\ref{prop:pathwise}). Our main result
is that, under any probability measure for which the standing assumption
holds, they extend to a \emph{stochastic derivative}
$\nabla^{\Pp}_{X,J}=(\nabla^{\Pp}_X,\nabla^{\Pp}_J)$ that identifies the
integrands of every square-integrable martingale:
\begin{equation}\label{eq:intro-MRF}
Y_t=Y_0+\int_0^t(\nabla^{\Pp}_X Y)_s\cdot\mathrm dX_s
+\int_{(0,t]\times E}\nabla^{\Pp}_J Y(s,z)\,(J-\mu)(\mathrm ds,\mathrm dz)
+N_t .
\end{equation}

Our main theoretical contributions may be summarized as follows:
\begin{enumerate}
\item \emph{Construction of the stochastic derivative}
(Theorem~\ref{thm:closure}). The pathwise derivative is closable, and its
closure is an isometry from the closed space of stochastic integrals onto the
space of predictable integrands, whose inverse is the stochastic integral. The
core consists of integral functionals of bounded predictable fields
(Lemma~\ref{lem:density}), so no regularity of functionals in time or in the
path is required. The compensator may be random and path-dependent and need
not be absolutely continuous in time (Remark~\ref{rem:singular-time}), the
covariation $[X]$ may be degenerate, and the mark space is an arbitrary
standard Borel space.

\item \emph{Martingale representation formula}
(Theorem~\ref{thm:representation}). Extended by zero on martingales orthogonal
to all stochastic integrals, $\nabla^{\Pp}_{X,J}$ is the adjoint of the
stochastic integral, and \eqref{eq:intro-MRF} holds for every square-integrable
martingale, with the energy identity \eqref{eq:energy-decomposition}. The
representation property is not assumed: $N=0$ for all $Y$ exactly when it
holds. It holds, for example, for a Brownian motion and an independent Poisson
random measure, and for nonexplosive marked point processes, in their natural
filtrations \cite{Jacod75,JS03}; Remark~\ref{sec:extra-noise} exhibits the
natural filtration of a random measure in which it fails.

\item \emph{Pathwise computation} (Proposition~\ref{prop:identification}). If
$Y$ is given by a non-anticipative functional $F$ whose jump derivative is
predictable and which matches $Y_{t-}$ at the jump times, then
$\nabla^{\Pp}_J Y=\nabla_J F$: the jump integrand of $Y$ is the change in its
functional representative when a mark is inserted at the current time. No
smoothness of $F$, Malliavin differentiability or Markov structure is needed;
Proposition~\ref{prop:ito} treats the continuous component for regular
functionals. Section~\ref{sec:examples} uses this to obtain, without Malliavin
calculus, the jump integrand \eqref{eq:max-representation} of the running
maximum of a L\'evy martingale, found in \cite{RR11} through the Clark--Ocone
formula, and to compute martingale representations
\eqref{eq:hawkes-representation}--\eqref{eq:hawkes-terminal-count} for a
Hawkes process with a non-Markovian kernel directly from its event history.

\item \emph{Relation with Malliavin calculus} (Theorem~\ref{thm:lifting}). We show that, on
Wiener--Poisson space, the stochastic derivative is the predictable projection
\eqref{eq:projection-theorem} of the derivative $(D^W,\dadd)$ of Ishikawa and
Kunita \cite{IK06} and Kabanov and Skorokhod \cite{KS75}, and this projection extends continuously
to all of $L^2$. The Clark--Ocone formula on Wiener--Poisson space is thus a
special case of \eqref{eq:intro-MRF}. By contrast, the stochastic derivative
cannot in general be recovered from local gradients that displace existing
marks, such as the construction of Bismut \cite{Bismut83} or the lent particle method of Bouleau and Denis
\cite{BD09} (Proposition~\ref{prop:local} and Remark~\ref{rem:extra-mark}).
\end{enumerate}

\paragraph{Applications}
Sections~\ref{sec:compensator-perturbations}--\ref{sec:pathwise-mecke} use the
stochastic derivative in settings where neither the representation property nor
Malliavin differentiability is assumed.
Section~\ref{sec:compensator-perturbations} considers perturbations
$\mu\mapsto(1+\varepsilon h)\mu$ of the compensator, which may be random, by a
bounded predictable field $h$ with a deterministic bound on its energy
$\int h^2\,\mathrm d\mu$. For every square-integrable terminal variable $\xi$,
Theorem~\ref{thm:compensator-response} gives the sensitivity formula
\eqref{eq:perturb-first-response},
\[
\frac{\mathrm d}{\mathrm d\varepsilon}\,\E^{\varepsilon}\xi\,\Big|_{\varepsilon=0}
=\E\int_{(0,T]\times E}\nabla^{\Pp}_J(\CE\xi)\,h\,\mathrm d\mu,
\qquad
(\CE\xi)_t=\E[\xi\mid\mathcal F_t]-\E[\xi\mid\mathcal F_0],
\]
together with a convergent power series \eqref{eq:perturb-payoff-series} for
$\E^{\varepsilon}\xi$ whose coefficients and remainders are bounded explicitly.
Proposition~\ref{prop:perturb-change-law} shows how the stochastic derivative
itself transforms under an equivalent change of measure
\eqref{eq:perturb-transport-derivative}.
Section~\ref{sec:entropy} decomposes the relative entropy of an equivalent
measure $\mathbb Q$ into a compensator term and a remainder
\eqref{eq:entropy-decomposition}. The $\mathbb Q$-compensator of $J$ is $\mu(1+\nabla^{\Pp}_J Z/Z_-)$, where $Z$ is the density martingale, so the
compensator term is determined by the stochastic derivative; the remainder is
nonnegative and bounded above and below by multiples of the energy of the
orthogonal residual $N$ \eqref{eq:entropy-remainder-bounds}
(Theorem~\ref{thm:entropy-decomposition}). Consequently, the relative entropy
reduces to the compensator term for every density $\zeta$ with
$\E[\zeta\mid\mathcal F_0]=1$ that is bounded above and away from zero if and
only if $J-\mu$ has the predictable representation property
(Corollary~\ref{cor:entropy-representation}). To our knowledge, this
characterization of the representation property is new.
Section~\ref{sec:pathwise-mecke} proves a pathwise identity for exponential
functionals of a point configuration (Proposition~\ref{prop:mecke-pathwise}).
For a random measure with deterministic, time-atomless compensator, it shows
that the remainder in Mecke's formula \cite{Mecke67} is the terminal value of a
stochastic integral whose integrand is its stochastic derivative
\eqref{eq:mecke-full-pathwise}--\eqref{eq:mecke-remainder-derivative}. Mecke's
formula is then a consequence of this more general pathwise formula
(Corollary~\ref{cor:mecke-remainder}).

\paragraph{Relation to previous work}
Stochastic derivatives inverting the It\^o integral were introduced for
Brownian martingales by Isaacson \cite{Isaacson69} and Davis \cite{Davis75}.
 Kabanov and Skorokhod
\cite{KS75} constructed an anticipative differential operator for stochastic measures with independent values. In the Poisson case, the adjoint of the stochastic integral
is then the predictable projection of their derivative; see
 Section~\ref{sec.WienerPoisson}.
Di Nunno and Rozanov \cite{DNR07} define a derivative as an adjoint of the stochastic integral for martingale integrators and for stochastic
measures with independent values.

This adjoint is precisely the stochastic derivative operator
 $\nabla^{\Pp}_{X,J}$  
(Theorem~\ref{thm:representation}); what is new is its representation as
the closure of an explicit pathwise operator on functionals, for
integer-valued random measures with general predictable compensators. This allows a pathwise computation (Section~\ref{sec:pathwise-computation}) and enables the applications of
Sections~\ref{sec:compensator-perturbations}--\ref{sec:pathwise-mecke}. Using the functional approach of  \cite{CF13}, we construct the weak vertical derivative for continuous
martingales, give its counterpart for integer-valued random measures, treat
the continuous and jump components in a single energy space, and use a core of
integral functionals that does not require horizontal differentiability. Unlike the
point-process formulas cited above, the jump derivative requires neither the
natural filtration nor an enumeration of jump times; in contrast to Clark--Ocone
formulas, no anticipative
differentiation is involved.

\paragraph{Outline}
Section~\ref{sec:functionals} introduces non-anticipative functionals of
$(X,J,\mu,[X])$ and their pathwise derivatives. Section~\ref{sec:stochastic}
constructs the stochastic derivative, proves the martingale representation
formula, and gives criteria and examples for computing the derivative
pathwise. Section~\ref{sec:malliavin} discusses the relation with Malliavin
calculi on Wiener--Poisson space. Sections~\ref{sec:compensator-perturbations},
\ref{sec:entropy} and~\ref{sec:pathwise-mecke} contain the applications to
compensator perturbations, relative entropy and Mecke's formula.
In Section \ref{sec:entropy} we obtain a decomposition for the relative entropy of integer-valued random measures  using the stochastic derivative.
Finally, Section~\ref{sec:pathwise-mecke} derives a pathwise identity
implying the Mecke formula.
\section{Functionals of integer-valued measures}\label{sec:functionals}

\subsection{Stopped histories and non-anticipative functionals}
For definitions and generalities on stochaastic processes and integer-valued random measures we refer to \cite{ikeda2014,JS03}.

Fix a finite horizon $T$ and a standard Borel mark space $(E,\mathcal E)$.
We use a measurable history space with coordinates
$\omega=(x,j,a,q)$, where $x\in D([0,T],\R^d)$, $j$ is a
$\sigma$-finite integer-valued measure on $(0,T]\times E$ satisfying
$j(\{t\}\times E)\leq1$, $a$ is a time-atomless positive measure,
and $q$ is a continuous matrix-valued increasing path with $q_0=0$.
The measure coordinates carry their evaluation $\sigma$-fields;
integration of measurable fields against them is assumed measurable.
An initial coordinate may be included and is suppressed in the
notation. Further observed history coordinates may also be included;
they are held fixed under the perturbations. Under the probability law introduced in
Section~\ref{sec:stochastic}, $x=X$ is continuous, $j=J$, $a=\mu$,
and $q=[X]$. The larger c\`adl\`ag space for $x$ permits vertical
perturbations.

For a measure $\mu$, denote $\mu^t=\mu|_{(0,t]\times E}$ its restriction to $(0,t]\times E$ and
$\mu^{t-}=\mu|_{(0,t)\times E}$. For a path $x$, write
$x^t(s)=x(s\wedge t)$ and, for $t>0$,  $$x^{t-}=x\ 1_{[0,t)}+ x(t-)\ 1_{[t,T]}.$$ In particular, $x^t=x^{t-}$ if $x$ is continuous at $t$. We denote $\omega^t=(x^t,j^t,a^t,q^t)$.
All auxiliary coordinates are retained under the perturbations below.

\begin{definition}
A non-anticipative functional is a jointly measurable map
$F:[0,T]\times\Omega\to\R$ such that
$F(t,\omega)=F(t,\omega^t)$. A predictable functional field
$\psi(t,z,\omega)$ is measurable for the raw predictable
$\sigma$-field and $\mathcal E$; we use  
representatives satisfying
\begin{equation}\label{eq:strict-past}
 \psi(t,z,\omega)
  =\psi(t,z,x^{t-},j^{t-},a^{t-},q^{t-}).
\end{equation}
We use the same convention for vector-valued predictable functionals.
\end{definition}

The raw filtration is generated by the coordinate histories. Its
usual augmentation will be used for stochastic integration. Predictable
fields in the augmented filtration admit raw predictable versions
modulo the control measures used below. One can verify this on a
generator $\xi\one_{(s,u]}$: choose raw
$\mathcal F^0_{s+1/n}$-measurable versions $\xi_n$ of $\xi$, and use
$\xi_n\one_{(s+1/n,u]}$. Outside one null set these converge to the
generator; a monotone-class argument gives the assertion for fields.

\subsection{Pathwise derivatives: vertical differentiation and mark insertion}

For $v\in\R^d$, define the {\it vertical} perturbation
$x^{t,v}=x^t+v\one_{[t,T]}$. The vertical derivative \cite{Dupire09,CF10} is
\begin{equation}\label{eq:vertical}
 \nabla_xF(t,x,j,a,q)
=\left.\nabla_vF(t,x^{t,v},j^t,a^t,q^t)\right|_{v=0},
\end{equation}
whenever this derivative exists. Its predictable evaluation is denoted
\begin{equation}\label{eq:nablaX}
 \nabla_XF(t,\omega)
   =\nabla_xF(t,x^{t-},j^{t-},a^t,q^t).
\end{equation}

\begin{definition}[Jump derivative]\label{def:jump}
The  jump derivative of $F$ is defined as
\begin{equation}\label{eq:nablaJ}
 \nabla_JF(t,z,\omega)
=F(t,x^{t-},j^{t-}+\delta_{(t,z)},a^t,q^t)
       -F(t,x^{t-},j^{t-},a^t,q^t),
\end{equation}
whenever both terms are finite. We use this derivative on functionals
for which the resulting field is predictable.
\end{definition}

Insertion into $j^{t-}$ preserves the one-mark-per-time condition,
including at an actual jump time. Keeping $a$ fixed specifies the
perturbation without requiring a new compensator off the support of
the law. This convention concerns the stopped functional; it makes
no assertion about the conditional law of its future.

Set $F^-(t,\omega)=F(t,x^{t-},j^{t-},a^t,q^t)$. Directly from
\eqref{eq:nablaJ},
\begin{align}
 \nabla_J(FG)
 &=F^-\nabla_JG+G^-\nabla_JF+\nabla_JF\,\nabla_JG,
                                                     \label{eq:product}\\
 \nabla_J\Phi(F)
 &=\Phi(F^-+\nabla_JF)-\Phi(F^-).                    \label{eq:chain}
\end{align}
These identities require only the indicated values to be finite.
For the continuous derivative the ordinary first-order product and
chain rules hold whenever the vertical derivatives exist.

\subsection{Integral functionals}

For a finite-valued predictable field $\psi$, define
\begin{equation}\label{eq:jump-functional}
 F_\psi(t,\omega)
  =\int_{(0,t]\times E}\psi(s,z,\omega)\,j(\mathrm ds,\mathrm dz)
    -\int_{(0,t]\times E}\psi(s,z,\omega)\,a(\mathrm ds,\mathrm dz)
\end{equation}
on the histories where both integrals are absolutely convergent.
Choose a measurable extension, for instance zero, elsewhere.

For the continuous component let
\begin{equation}\label{eq:tau}
 \tau_n(q)=\inf\{t:\tr q_t\geq n\}\wedge T.
\end{equation}
An elementary predictable vector field has the form
\begin{equation}\label{eq:phi-core}
 \phi_r=\xi\one_{(s,u]}(r)\one_{\{r\leq\tau_n(q)\}}e_i,
 \qquad 0\leq s<u\leq T,
\end{equation}
where $\xi$ is bounded and measurable with respect to the raw history
at $s$. Its integral functional is, for $t>s$,
\begin{equation}\label{eq:continuous-functional}
 F_\phi(t,\omega)=\xi(\omega^s)
 \bigl[x^i(t\wedge u\wedge\tau_n(q))
          -x^i(s\wedge\tau_n(q))\bigr],
\end{equation}
and is zero for $t\leq s$. The minima make this a non-anticipative
functional. Extend this definition to finite sums by linearity.

\begin{proposition}[Inversion on integral functionals]\label{prop:pathwise}
For the above functionals, on their admissible histories,
\begin{equation}\label{eq:core-identities}
 \nabla_XF_\phi=\phi,\qquad \nabla_JF_\phi=0,
 \qquad \nabla_XF_\psi=0,\qquad \nabla_JF_\psi=\psi.
\end{equation}
The continuous identities are evaluated on continuous $x$; time
boundary conventions are immaterial in the continuous energy norm.
\end{proposition}

\begin{proof}
In \eqref{eq:continuous-functional}, an endpoint change affects the
first value of $x$ precisely when $s<t\leq u\wedge\tau_n(q)$.
The coefficient is then $\xi e_i$. The history at $s$ is unchanged,
as is $\tau_n(q)$. Inserting a mark at $t>s$ does not change either
quantity; for $t\leq s$ the functional is zero.

For \eqref{eq:jump-functional}, all integrand values at times
$r\leq t$ are unchanged by either perturbation, by
\eqref{eq:strict-past}. Insertion changes the first integral by
$\psi(t,z,\omega)$ and leaves the second unchanged. Vertical
perturbation of $x$ changes neither integral. Removing or adding
one endpoint mark preserves absolute convergence, since $\psi$
is finite there.
\end{proof}

The integral representatives in this proposition determine the initial
operator. Its independence of representatives as a stochastic operator
will follow from the isometry and closure theorem.

\section{Stochastic derivative and martingale representation formula}
\label{sec:stochastic}
\subsection{Stochastic integrals and functional representations}

Let $(\Omega,\mathcal F,\F,\Pp)$ be a filtered probability space satisfying the usual conditions. Let
$X$ be a continuous $\R^d$-valued local martingale, $X_0=0$, and let
$J$ have predictable compensator $\mu$. We denote
$\Jt=J-\mu$. An initial $\sigma$-field $\mathcal F_0$ is allowed.
For functional representations we assume that $\F$ is the usual
augmentation of the coordinate history filtration in
Section~\ref{sec:functionals}.
The Hilbert-space statements also apply on an arbitrary filtered
space with these integrators.

\begin{assumption}\label{ass:main}
Outside a $\Pp$-null set, we have
\begin{equation}\label{eq:time-assumptions}
 \forall t\in (0,T],\qquad J(\{t\}\times E)\leq1,\qquad
 \mu(\{t\}\times E)=0.
\end{equation}The predictable control measure
\begin{equation}\label{eq:control}
 m(B)=\E\int_{(0,T]\times E}\one_B(\omega,t,z)
                           \mu(\omega;\mathrm dt,\mathrm dz),
 \qquad B\in\cP\otimes\mathcal E,
\end{equation}
is $\sigma$-finite.
\end{assumption}

Choose predictable $B_n\uparrow\Omega\times(0,T]\times E$ modulo
$m$, with $m(B_n)<\infty$. Compensation gives
\begin{equation}\label{eq:compensation}
 \E\int H\,\mathrm dJ=\E\int H\,\mathrm d\mu
 \quad\text{for all nonnegative predictable fields }H.
\end{equation}
In particular the control measure is a measure on predictable
fields; it need not have a density with respect to
$\Pp\otimes\mathrm dt$ and a fixed mark measure.

Define
\begin{align}
 \HJ&=L^2(\Omega\times(0,T]\times E,\cP\otimes\mathcal E,m),
       &\norm{\psi}_{\HJ}^2&=\E\int|\psi|^2\,\mathrm d\mu,
                                                       \label{eq:HJ}\\
 \HX&=\left\{\phi\text{ predictable}:\E\int_0^T
             \phi_t^\top\mathrm d[X]_t\,\phi_t<\infty\right\}/\!\sim,
&\norm{\phi}_{\HX}^2&=\E\int_0^T\phi_t^\top\mathrm d[X]_t\,\phi_t, \label{eq:HX}
\end{align}
where $\phi\sim\phi'$ if their difference has zero energy. Put
$\HXJ=\HX\oplus\HJ$. Martingales starting from zero form the
Hilbert space $\cM$ with norm $\norm{Y}_{\cM}^2=\E[Y_T^2]$;
they are identified up to indistinguishability.

The  quotient space in \eqref{eq:HX} is complete.  $[X]$ is
absolutely continuous with respect to the increasing process
$\tr[X]$, so $\mathrm d[X]_t=c_t\,\mathrm d\tr[X]_t$
for a predictable process $c$ with values in the symmetric
nonnegative-definite $d\times d$ matrices. Since
$\phi^\top\mathrm d[X]\,\phi=|c^{1/2}\phi|^2\,\mathrm d\tr[X]$,
the map $\phi\mapsto c^{1/2}\phi$ identifies $\mathcal L^2_{\Pp}([X])$ with
the closed subspace of
$L^2\bigl(\Omega\times(0,T],\mathcal P,\Pp\otimes\mathrm d\tr[X];\mathbb R^d\bigr)$
of fields with values in $\operatorname{Ran}c^{1/2}$. The Moore--Penrose inverse is a Borel function of the matrix, since
$(c^{1/2})^{+}=\lim_{\varepsilon\downarrow0}(c+\varepsilon I)^{-1}c^{1/2}$.
Hence, for a field $\psi$ with values in $\operatorname{Ran}c^{1/2}$, the
field $\phi=(c^{1/2})^{+}\psi$ is predictable and satisfies
$c^{1/2}\phi=\psi$. Thus degeneracy of $[X]$
causes no ambiguity beyond the equivalence relation.

The integral operator is
\begin{equation}\label{eq:integral}
 I_{X,J}:\HXJ\longrightarrow\cM,
 \qquad I_{X,J}(\phi,\psi)=\phi\cdot X+\psi*\Jt.
\end{equation}
Denote $I_J\psi=\psi*\Jt$ the integral against the compensated random measure, $I_X\phi=\phi\cdot X$, and 
\begin{equation}\label{eq:range}
 \IXJ=\Ran I_{X,J},\qquad \cIJ=\Ran I_J.
\end{equation}
\begin{lemma}[Isometry]\label{lem:isometry}
Under Assumption~\ref{ass:main},
\begin{equation}\label{eq:isometry}
 \norm{I_{X,J}(\phi,\psi)}_{\cM}^2
     =\norm{\phi}_{\HX}^2+\norm{\psi}_{\HJ}^2.
\end{equation}
Moreover $I_J\psi$ is purely discontinuous and
\begin{equation}\label{eq:jump-bracket}
 \Delta(I_J\psi)_t=\int_E\psi(t,z)J(\{t\},\mathrm dz),
 \qquad [I_J\psi]_t=\int_{(0,t]\times E}|\psi|^2\,\mathrm dJ.
\end{equation}
\end{lemma}

\begin{proof}
For bounded integrands of finite control-measure support, the
compensated integrals are absolutely convergent. Atomlessness of
$\mu$ gives the first identity in \eqref{eq:jump-bracket}.
The one-mark-per-time condition gives the second, since the square
of each jump equals its contribution to $\psi^2*J$. Compensation
and localization yield
$\E[(I_J\psi)_T^2]=\E\int\psi^2\,\mathrm d\mu$.
The usual $L^2$ construction extends the integral and these
identities to $\HJ$; see \cite[Chapter II]{JS03}.
Its continuous covariation with $I_X\phi$ vanishes. The continuous
It\^o isometry therefore proves \eqref{eq:isometry}.
\end{proof}

Let $\mathcal S_J$ be the vector space of bounded predictable fields
supported on a predictable set of finite $m$-measure. Let
$\mathcal S_X$ be the span of \eqref{eq:phi-core}, with raw
predictable versions, and put
$\mathcal S=\mathcal S_X\oplus\mathcal S_J$.

\begin{lemma}[Density of integral functionals]\label{lem:density}
The space $\mathcal S$ is dense in $\HXJ$. For
$\psi\in\mathcal S_J$,
\begin{equation}\label{eq:absolute}
 \E\int|\psi|\,\mathrm dJ
     =\E\int|\psi|\,\mathrm d\mu<\infty.
\end{equation}
Consequently every $I_{X,J}(\phi,\psi)$ with
$(\phi,\psi)\in\mathcal S$ has the functional representative
$F_\phi+F_\psi$ of Section~\ref{sec:functionals}.
\end{lemma}

\begin{proof}
For $\psi\in\HJ$ choose a finite predictable version and set
\begin{equation}\label{eq:truncation}
 \psi_n=\one_{B_n}\max(-n,\min(\psi,n)).
\end{equation}
Then $\psi_n\in\mathcal S_J$ and
$|\psi_n-\psi|^2\leq|\psi|^2$, so dominated convergence proves
$\psi_n\to\psi$ in $\HJ$. If $\psi$ is supported on $B$ and
bounded by $K$, \eqref{eq:compensation} gives
$\E\int|\psi|\,\mathrm d(J+\mu)\leq2K m(B)$.
Finiteness on $(0,T]$ implies simultaneous absolute convergence on
all stopped intervals.

For the continuous part, first replace $\phi$ by
$\one_{\{t\leq\tau_n([X])\}}\one_{\{|\phi_t|\leq n\}}\phi_t$.
These truncations converge in the matrix energy norm. Continuity of
$[X]$ gives $\E\tr [X]_{\tau_n}\leq n$, so the truncated field belongs
to a finite trace-control $L^2$ space. Bounded elementary predictable
fields are dense in that space. Stopping their elementary
approximations at $\tau_n$ gives elements of $\mathcal S_X$.
Since $v^\top\mathrm d[X]\,v\leq|v|^2\mathrm d\tr [X]$, convergence
in the trace-control norm implies convergence in \eqref{eq:HX}.
A diagonal choice completes the proof.
\end{proof}

\subsection{Closure of the pathwise derivative}
On $\mathcal D_0=I_{X,J}(\mathcal S)$ define  the \emph{pathwise derivative}: 
\begin{equation}\label{eq:initial-operator}
 \nabla^0_{X,J}Y
    =\bigl(\nabla_X(F_\phi+F_\psi),
            \nabla_J(F_\phi+F_\psi)\bigr)
    =(\phi,\psi),
 \qquad Y=I_{X,J}(\phi,\psi).
\end{equation}
$\nabla^0_{X,J}Y$ is
computed path by path from the functional representative, without
reference to $\Pp$.
The equalities follow from Proposition~\ref{prop:pathwise}.
We now extend this operator to the entire  stochastic integral range:
\begin{theorem}[Closure of the pathwise derivative]\label{thm:closure}
Under Assumption~\ref{ass:main}, $\IXJ$ is a closed subspace of
$\cM$ and $\mathcal D_0$ is dense in $\IXJ$.
The operator \eqref{eq:initial-operator} is well defined in
$\HXJ$ and closable in $\cM\times\HXJ$. Its graph closure has
domain exactly $\IXJ$ and defines a bijective isometry
\begin{equation}\label{eq:closed-derivative}
 \nablap:\IXJ\longrightarrow\HXJ,
 \qquad \nablap=I_{X,J}^{-1}.
\end{equation}
In particular,
\begin{equation}\label{eq:inverse}
 \nablap I_{X,J}(\phi,\psi)=(\phi,\psi),\qquad
 I_{X,J}\nablap Y=Y.
\end{equation}
For $Y\in\IXJ$ and $(\phi,\psi)\in\HXJ$,
\begin{align}
 \E[Y_T I_{X,J}(\phi,\psi)_T]
 &=\E\int_0^T(\nabla_X^{\Pp}Y)_t^\top
\mathrm d[X]_t\,\phi_t
+\E\int_{(0,T]\times E}(\nabla_J^{\Pp}Y)(t,z)\psi(t,z)\,\mu(\mathrm dt,\mathrm dz).
  \label{eq:duality-range}
\end{align}
We call $\nablap$ the {\rm stochastic derivative}.
\end{theorem}

\begin{proof}
If $I_{X,J}h_n$ is Cauchy in $\cM$, the isometry makes $h_n$
Cauchy in the complete space $\HXJ$. Its limit $h$ satisfies
$I_{X,J}h_n\to I_{X,J}h$. Thus the range is closed; its stated
dense subspace follows from Lemma~\ref{lem:density}.

If two integral representatives induce the same martingale,
\eqref{eq:isometry} shows that their integrands agree in $\HXJ$.
This proves that \eqref{eq:initial-operator} is well defined and
that
\begin{equation}\label{eq:derivative-isometry}
 \norm{\nabla^0_{X,J}Y}_{\HXJ}=\norm{Y}_{\cM},
 \qquad Y\in\mathcal D_0.
\end{equation}
In particular, $Y_n\to0$ and $\nabla^0_{X,J}Y_n\to h$ imply
$h=0$.

More precisely, if
$(Y_n,\nabla^0_{X,J}Y_n)\to(Y,h)$, then
$Y_n=I_{X,J}\nabla^0_{X,J}Y_n$ implies $Y=I_{X,J}h$.
Conversely, for $Y=I_{X,J}h$, choose $h_n\in\mathcal S$ converging
to $h$. The graph sequence
$(I_{X,J}h_n,\nabla^0_{X,J}I_{X,J}h_n)$ converges to $(Y,h)$.
The closed graph is therefore exactly
$\{(I_{X,J}h,h):h\in\HXJ\}$. This proves
\eqref{eq:closed-derivative}--\eqref{eq:inverse}, and polarization
of \eqref{eq:isometry} gives \eqref{eq:duality-range}.
\end{proof}

\begin{corollary}[Approximation of the derivative]\label{cor:approximation}
If $Y_n\in\mathcal D_0$ and $Y_n\to Y$ in $\cM$, then
$Y\in\IXJ$ and $\nabla^0_{X,J}Y_n\to\nablap Y$ in $\HXJ$,
independently of the approximating sequence. Moreover,
\begin{equation}\label{eq:maximal}
 \E\sup_{t\leq T}|Y_n(t)-Y(t)|^2
       \leq4\norm{\nabla^0_{X,J}Y_n-\nablap Y}_{\HXJ}^2.
\end{equation}
\end{corollary}
\begin{proof}
Use \eqref{eq:derivative-isometry}, completeness, and Doob's
$L^2$ maximal inequality.
\end{proof}

\begin{remark}
The theorem includes the purely discontinuous case $X=0$, giving
$\nabla_J^{\Pp}:\cIJ\to\HJ$. For $J=0$ the inverse
characterization agrees with the weak vertical derivative of
\cite{CF13} wherever that construction applies. The functional core
used here requires no horizontal differentiability. Closure is in
the martingale and integrand norms, rather than in a topology on all
stopped histories.
\end{remark}

\subsection{Martingale representation formula and orthogonal decomposition}
Let $P_{X,J}$ be the orthogonal projection of $\cM$ onto $\IXJ$.
We extend the stochastic derivative to a linear operator on  the whole space
$\mathcal M^2_0$ by setting
\begin{equation}\label{eq:weak-definition}
\nabla^{\Pp}_{X,J}N=0 \qquad {\rm if}\qquad 
N\in\mathcal M^2_0\quad
 \text{ with } N\perp\IXJ.
\end{equation}
Then $$\nabla^{\Pp}_{X,J}Y=\nabla^{\Pp}_{X,J}(P_{X,J}Y).$$
For a square-integrable martingale $Y$ with
$Y_0\neq0$ we write $\nabla^{\Pp}_{X,J}Y$ for
$\nabla^{\Pp}_{X,J}(Y-Y_0)$.
\begin{theorem}[Martingale representation formula]\label{thm:representation}
The operator $\nablap$ is the Hilbert-space adjoint $I_{X,J}^*$ of the stochastic integral.
Every square-integrable martingale $Y$ has the unique decomposition
\begin{equation}\label{eq:decomposition}
 Y_t=Y_0+\int_0^t(\nabla_X^{\mathbb P}Y)_s\cdot\mathrm dX_s
       +\int_{(0,t]\times E}\nabla_J^{\mathbb P}Y(s,z)\Jt(\mathrm ds,\mathrm dz)+N_t,
\end{equation}
with $(\nabla_X^{\mathbb P}Y,\nabla_J^{\mathbb P}Y)\in L^2_{\mathcal P}([X],\mu)$, $N\in\mathcal M^2_0$ and
$N\perp\IXJ$.  In addition,
\begin{equation}\label{eq:energy-decomposition}
 \E[(Y_T-Y_0)^2]
   =\norm{\nablap Y}_{\HXJ}^2+\E[N_T^2].
\end{equation}
If $(X,J-\mu)$ has the square-integrable predictable representation
property in $\F$, then $\IXJ=\cM$ and
\begin{equation}\label{eq:representation}
Y_t=Y_0+\int_0^t(\nabla_X^{\Pp}Y)_s\cdot\mathrm dX_s
       +\int_{(0,t]\times E}\nabla_J^{\Pp}Y(s,z)  \Jt(\mathrm ds,\mathrm dz).
\end{equation}
Conversely, validity of \eqref{eq:representation} for every
square-integrable martingale is equivalent to that property.
\end{theorem}

\begin{proof}
The isometry gives $I_{X,J}^*I_{X,J}=\mathrm{Id}$ and
$I_{X,J}I_{X,J}^*=P_{X,J}$. Hence
$I_{X,J}^*=I_{X,J}^{-1}P_{X,J}=\nablap$.
Orthogonal projection now gives
\eqref{eq:decomposition} and \eqref{eq:energy-decomposition}.
The residual vanishes for every martingale precisely when $\IXJ=\mathcal M^2_0$ i.e. the integral
range is the whole space.
\end{proof}

The residual is strongly orthogonal to all integrals in $\IXJ$.
Indeed, testing terminal orthogonality with integrands multiplied by
$\one_A\one_{(s,t]}$, $A\in\mathcal F_s$, gives conditional
orthogonality of increments. The martingale property then shows
that the product of $N$ with each such integral is a martingale.

For $\xi\in L^2(\mathcal F_T)$ denote $\CE\xi$ the martingale defined by
\begin{equation}\label{eq:terminal-map}
 (\CE\xi)_t=\E[\xi\mid\mathcal F_t]-\E[\xi\mid\mathcal F_0],
 \qquad \mathsf D_{X,J}\xi=\nablap\CE\xi.
\end{equation}
Thus $\xi=\E[\xi\mid\mathcal F_0]+(I_{X,J}\mathsf D_{X,J}\xi)_T+N_T$,
and
\begin{equation}\label{eq:adjoint-terminal}
 \E[\xi(I_{X,J}h)_T]
     =\langle\mathsf D_{X,J}\xi,h\rangle_{\HXJ},
 \qquad h\in\HXJ.
\end{equation}
Under the representation property, this gives the terminal version
of \eqref{eq:representation} for every $\xi\in L^2$.

\begin{remark}\label{rem:singular-time}
The representation property is an assumption on the filtration and
the integrators. Conditions \eqref{eq:time-assumptions} establish
the isometry but do not by themselves establish this property,
even in the natural filtration of $J$; an example is given in
Remark~\ref{sec:extra-noise}. An atomless compensator may also be
singular in time: if $C$ is a deterministic continuous increasing
function and $N$ is a unit-rate Poisson process, the measure of
$N_{C(t)}$ has compensator $\mathrm dC(t)\,\delta_1(\mathrm dz)$.
All preceding results apply when $C$ is the Cantor function.
\end{remark}

\begin{remark}[A nonzero orthogonal residual in a natural filtration]
\label{sec:extra-noise}
{\rm An infinite-activity measure can carry information about continuous
noise. The following example, of the type considered in
\cite[Section 6]{Cohen13}, shows why the representation hypothesis
in Theorem~\ref{thm:representation} is separate from
Assumption~\ref{ass:main}.
Let $W$ be Brownian motion and construct, by independent Poisson
thinning, counting processes $N^n$ with intensities
$2^n e^{W_t}\,\mathrm dt$. Conditional on $W$, they are
independent Poisson processes with these time-dependent intensities.
The construction is non-anticipative, so $W$ remains Brownian in
the joint filtration. Define
\begin{equation}\label{eq:counterexample}
 J(\mathrm dt,\mathrm dz)=\sum_{n\geq1}\mathrm dN^n_t\,
                              \delta_{2^{-n}}(\mathrm dz),\qquad
 \mu(\mathrm dt,\mathrm dz)=\sum_{n\geq1}2^n e^{W_t}\,\mathrm dt\,
                              \delta_{2^{-n}}(\mathrm dz).
\end{equation}
A countable union argument using diffuse conditional jump times
shows time simplicity. The compensator is atomless, the control
measure is $\sigma$-finite, and
\[
 \E\int z^2\,\mathrm d\mu
       =\sum_{n\geq1}2^{-n}\int_0^T e^{t/2}\,\mathrm dt<\infty.
\]
Let $A_t=\int_0^t e^{W_s}\,\mathrm ds$. Conditional variances give
\[
 \E[(2^{-n}N^n_t-A_t)^2]=2^{-n}\E A_t.
\]
Chebyshev's inequality and Borel--Cantelli imply convergence to
$A_t$, simultaneously at rational times. Continuity recovers the
whole past of $A$ from the past of $J$, and, for $t>0$,
\[
 W_t=\log\left(\lim_{h\downarrow0}
                       \frac{A_t-A_{t-h}}{h}\right).
\]
Thus $W$ is adapted to the usual natural filtration of $J$ and is
Brownian there. The expression for $\mu$ in
\eqref{eq:counterexample} is predictable in this smaller filtration,
so remains its compensator. Every integral against $J-\mu$ is
purely discontinuous and is orthogonal to $W$. In the case $X=0$,
Theorem~\ref{thm:representation} therefore gives
$\nabla_J^{\mathbb P}W=0$ and residual $N=W\ne0$.}
\subsection{Pathwise computation of the stochastic derivative}\label{sec:pathwise-computation}

\begin{proposition}[Identification by jumps]\label{prop:identification}
Let $Y_t=F(t,X,J,\mu,[X])$ be a square-integrable c\`adl\`ag
martingale. Suppose $g=\nabla_JF$ is a finite predictable field
and, outside a null set, at every event time $t$ of $J$,
\begin{equation}\label{eq:compatibility}
 F(t,X^{t-},J^{t-},\mu^t,[X]^t)=Y_{t-}.
\end{equation}
Then $g\in\HJ$ and
\begin{equation}\label{eq:jump-identification}
 \nabla_J^{\mathbb P}Y=g.
\end{equation}
If $Y-Y_0\in\cIJ$, then
$Y_t=Y_0+\int_{(0,t]\times E}g(s,z)\,\Jt(\mathrm ds,\mathrm dz)$.
\end{proposition}

\begin{proof}
At an atom $(t,z)$ of $J$, continuity of $X$ and the absence of endpoint mass in $\mu$, together with
\eqref{eq:compatibility} give $\Delta Y_t=g(t,z)$. Therefore
\begin{equation}\label{eq:jump-energy}
 \E\int g^2\,\mathrm d\mu
   =\E\int g^2\,\mathrm dJ
   \leq\E[[Y-Y_0]_T]=\E[(Y_T-Y_0)^2].
\end{equation}
Compensation applies first to nonnegative fields, so this also
proves integrability. For $h\in\HJ$, pure discontinuity and
\eqref{eq:jump-bracket} give
\[
 \E[(Y_T-Y_0)(I_Jh)_T]
   =\E[[Y-Y_0,I_Jh]_T]
   =\E\int gh\,\mathrm dJ
   =\E\int gh\,\mathrm d\mu.
\]
Cauchy--Schwarz justifies the integrable covariations. The jump
component of the adjoint is $I_J^*$, proving
\eqref{eq:jump-identification}. Membership in $\cIJ$ removes the
residual and the continuous component.
\end{proof}
Compatible functionals representing the same martingale have
the same jump derivative $m$-almost everywhere. Compatibility is
required only to identify a given functional representative.
\begin{example}[Stochastic exponentials]\label{ex:exponential}
Let $\theta\in\mathcal S_J$ with $\theta\geq-1+\varepsilon$ for some
$\varepsilon>0$, and define
\begin{equation}\label{eq:exp-functional}
 F(t,j,a)=\exp\left\{
      \int_{(0,t]\times E}\log(1+\theta(s,z))\,j(\mathrm ds,\mathrm dz)
       -\int_{(0,t]\times E}\theta(s,z)\,a(\mathrm ds,\mathrm dz)
                  \right\},
\end{equation}
where the predictable dependence of $\theta$ on the history is
understood. Both integrals converge absolutely almost surely, and
$Z_t=F(t,J,\mu)$ is the stochastic exponential of $\theta*\Jt$:
\begin{equation}\label{eq:exp-SDE}
 Z_t=1+\int_{(0,t]\times E}Z_{s-}\theta(s,z)  \Jt(\mathrm ds,\mathrm dz).
\end{equation}
Inserting a mark at $(t,z)$ multiplies $F$ by $1+\theta(t,z)$, so
\begin{equation}\label{eq:exp-derivative}
 \nabla_JF(t,z,J,\mu)=Z_{t-}\theta(t,z),
\end{equation}
in agreement with the chain rule \eqref{eq:chain}. Since
$F(t,J^{t-},\mu^t)=Z_{t-}$, the compatibility condition
\eqref{eq:compatibility} holds. If $Z$ is a square-integrable
martingale, for instance when $\int\theta^2\,\mathrm d\mu\leq K$
almost surely for a deterministic $K$
(Theorem~\ref{thm:compensator-response}(i)), then
Proposition~\ref{prop:identification} gives $Z_-\theta\in\HJ$, and
\eqref{eq:exp-SDE} with Theorem~\ref{thm:closure} yields
\[
 \nabla_J^{\Pp}(Z-1)=Z_-\theta.
\]
Thus the pathwise derivative of $F$ coincides with the stochastic
derivative of $Z$. Neither a deterministic compensator nor
independent increments is used.

\end{example}
\begin{proposition}[Identification via functional It\^o calculus]
\label{prop:ito}
Suppose a square-integrable martingale $Y$, given by a
non-anticipative functional, admits a valid functional It\^o
decomposition
\begin{equation}\label{eq:ito-form}
 Y_t-Y_0=B_t+\int_0^t\beta_s\cdot\mathrm dX_s
              +\int_{(0,t]\times E}\gamma(s,z)\,\Jt(\mathrm ds,\mathrm dz),
\end{equation}
where $(\beta,\gamma)\in\HXJ$ and $B$ is continuous, adapted,
locally of finite variation, with $B_0=0$. Then $B=0$,
$Y-Y_0\in\IXJ$, and $\nablap Y=(\beta,\gamma)$.
\end{proposition}
\begin{proof}
The difference in \eqref{eq:ito-form} is a continuous local
martingale of locally finite variation, hence is zero.
Theorem~\ref{thm:closure} identifies its integrands.
\end{proof}

For example, let $E\subset\R^d\setminus\{0\}$,
$\E\int|z|^2\,\mathrm d\mu<\infty$, and
$S=S_0+X+z*\Jt$. For a functional $F$ satisfying the
hypotheses of the c\`adl\`ag functional It\^o formula
\cite[Proposition 6]{CF10}, set
\begin{equation}\label{eq:ito-coefficients}
 \beta_t=\nabla_xF(t,S^{t-}),\qquad
 \gamma(t,z)=F(t,S^{t-}+z\one_{[t,T]})-F(t,S^{t-}).
\end{equation}
When these coefficients have finite energy and the following drift
terms are locally absolutely integrable, the finite-variation term
in \eqref{eq:ito-form} is
\begin{align}\label{eq:ito-drift}
 B_t={}&\int_0^t\mathcal D F(s,S^{s-})\,\mathrm ds
     +\frac12\int_0^t\tr\bigl(\nabla_x^2F(s,S^{s-})\,\mathrm d[X]_s\bigr)
       +\int_{(0,t]\times E}\bigl(\gamma(s,z)-\beta_s\cdot z\bigr)
\mu(\mathrm ds,\mathrm dz).
\end{align}
Here $\mathcal D$ denotes horizontal differentiation by constant
extension of the path. The atomless-compensator assumption makes
$B$ continuous. If $F(t,S^t)$ is a square-integrable martingale,
Proposition~\ref{prop:ito} gives its representation with
\eqref{eq:ito-coefficients}, without a global representation
assumption on the filtration.
\subsection{Examples}
\label{sec:examples}
\subsubsection{Running maximum of a pure-jump martingale}

Let $\eta$ be a Poisson measure on
$(0,T]\times(\R\setminus\{0\})$ with intensity
$\mathrm dt\,\nu(\mathrm dz)$, where
$\int z^2\,\nu(\mathrm dz)<\infty$, and use its natural
filtration. Set
\begin{equation}\label{eq:levy-max}
 L_t=\int_{(0,t]\times\R}z(\eta-\mathrm dt\,\nu)(\mathrm ds,\mathrm dz),
 \qquad S_t=\sup_{0\leq s\leq t}L_s.
\end{equation}
The supremum includes $L_0=0$. By Doob's inequality $S_T\in L^2$.
For $u\geq0$ let $S'_u$ have the law of the supremum over $[0,u]$
of an independent copy of $L$, and define
\begin{equation}\label{eq:max-value}
 H(t,x,m)=\E\bigl[m\vee(x+S'_{T-t})\bigr].
\end{equation}
Independent increments imply
$Y_t=\E[S_T\mid\mathcal F_t]=H(t,L_t,S_t)$.
Writing $\bar F_u(y)=\Pp(S'_u>y)$, so that $\bar F_u(y)=1$ for $y<0$,
we have
\begin{equation}\label{eq:max-value-tail}
 H(t,x,m)=m+\int_{m-x}^{\infty}\bar F_{T-t}(y)\,\mathrm dy;
\end{equation}
compare \cite[(3.1)]{RR11}.
The map $H$ is Lipschitz in $x,m$ and continuous in $t$; the latter
follows by the $L^2$ maximal inequality for increments of $L$.
To choose a functional version of $L$, truncate the marks away
from zero and infinity and then select a deterministic subsequence
of the compensated integral functionals converging uniformly
almost surely. Such a subsequence exists by the isometry and
Doob's inequality. The set of convergence is stable under any
single mark insertion or removal, since a fixed nonzero mark is
included in every sufficiently large truncation. On a stopped
history, use the uniform limit on that stopped interval, and set
the value to zero if it does not exist. This gives a non-anticipative
version with endpoint increment $z$ on the admissible histories.

An inserted jump $z$ changes $(x,m)$ to
$(x+z,m\vee(x+z))$. Since $S'_{T-t}\geq0$, the functional
jump difference is
\begin{equation}\label{eq:max-integrand}
 g(t,z)=\E\left[
       S_{t-}\vee(L_{t-}+z+S'_{T-t})
        -S_{t-}\vee(L_{t-}+S'_{T-t})
                 \right],
\end{equation}
where the expectation integrates only the independent copy.
By \eqref{eq:max-value-tail}, with $a_t=S_{t-}-L_{t-}\geq0$,
\begin{equation}\label{eq:max-integrand-explicit}
 g(t,z)=\int_{a_t-z}^{a_t}\bar F_{T-t}(y)\,\mathrm dy
 =\begin{cases}
   (z-a_t)+\displaystyle\int_0^{a_t}\bar F_{T-t}(y)\,\mathrm dy,
       & z\geq a_t,\\[6pt]
   \displaystyle\int_{a_t-z}^{a_t}\bar F_{T-t}(y)\,\mathrm dy,
       & 0\leq z<a_t,\\[6pt]
   -\displaystyle\int_{a_t}^{a_t-z}\bar F_{T-t}(y)\,\mathrm dy,
       & z<0,
  \end{cases}
\end{equation}
where the first integral is oriented. The field $g$ is predictable
and, since $0\leq\bar F\leq1$, $|g(t,z)|\leq|z|$, so
$g\in\mathcal L^2_{\Pp}(\mu)$.
The continuity properties of $H$ give the compatibility condition
\eqref{eq:compatibility}. Proposition~\ref{prop:identification}
and the Poisson representation property therefore yield
\begin{equation}\label{eq:max-representation}
 S_T=\E[S_T]+
     \int_{(0,T]\times\R}g(t,z)
                         (\eta-\mathrm dt\,\nu)(\mathrm dt,\mathrm dz).
\end{equation}

The coefficient \eqref{eq:max-integrand-explicit} coincides with the
jump integrand obtained by R\'emillard and Renaud
\cite[Theorem~2.2]{RR11} for the maximum of a square-integrable
L\'evy process. Their proof uses the Clark--Ocone formula of the
Malliavin calculus for L\'evy processes, and requires showing that
the maximum is Malliavin differentiable, by approximating it by
maxima over finite sets of times. Here the coefficient is obtained
directly from Proposition~\ref{prop:identification}, as the finite
difference of the conditional expectation functional under mark
insertion; no Malliavin or vertical differentiability of the maximum
is needed.

The same argument applies when $L$ has, in addition, a drift and a
Brownian component, as in \cite{RR11}:
Proposition~\ref{prop:identification} again identifies the jump
integrand as \eqref{eq:max-integrand-explicit}, with $S'$ the
supremum of the full process. The Brownian integrand found in
\cite{RR11} is not covered by
Proposition~\ref{prop:identification}.
\end{remark}

\subsubsection{A self-exciting jump process}
\label{subsec:hawkes-history}
Let $N$ be a linear Hawkes process, started with empty history at
time zero, in its usual natural filtration. Let $\nu>0$ and let
$\kappa:[0,\infty)\to[0,\infty)$ be bounded and continuous, with
\[
\|\kappa\|_1:=\int_0^\infty\kappa(u)\,\mathrm du<1.
\]
Its predictable intensity, jump measure, and compensator are
\begin{equation}\label{eq:hawkes-intensity}
 \lambda_t=\nu+\int_{(0,t)}\kappa(t-s)\,\mathrm dN_s,
 \qquad J(\mathrm dt,\mathrm dz)=\mathrm dN_t\,\delta_1(\mathrm dz),
 \qquad \mu(\mathrm dt,\mathrm dz)=\lambda_t\,\mathrm dt\,\delta_1(\mathrm dz).
\end{equation}
For example, $\kappa(u)=\alpha(1+u)^{-2}$, with $0<\alpha<1$,
depends on every past event through its age. The intensity is not a
function of the current count alone. We use the event history directly,
without a finite-dimensional Markovian reduction.

The Poisson cluster construction of Hawkes and Oakes
\cite{HawkesOakes74} gives a nonexplosive
time-simple process. Each immigrant initiates a family in which every
event produces children according to an independent Poisson process
with age intensity $\kappa$. The total family size has first moment
$(1-\|\kappa\|_1)^{-1}$ and second moment $(1-\|\kappa\|_1)^{-3}$; these follow from the
branching process with  Poisson$(\|\kappa\|_1)$ offspring. Consequently,
\begin{equation}\label{eq:hawkes-moments}
 \mathbb E N_T\leq\frac{\nu T}{1-\|\kappa\|_1},\qquad
 \mathbb E N_T^2\leq\frac{\nu T}{(1-\|\kappa\|_1)^3}
                         +\left(\frac{\nu T}{1-\|\kappa\|_1}\right)^2.
\end{equation}
Thus the compensator is time-atomless and its control measure is finite
on the given horizon.

\paragraph{Conditional values on histories.} For a finite simple history $n$ on $(0,t]$, define a continuation on
$(t,T]$ with intensity
\begin{equation}\label{eq:hawkes-continuation}
 \lambda_u^{t,n}
   =b_n(u)+\int_{(t,u)}\kappa(u-s)\,\mathrm dN_s^{t,n},
 \qquad
 b_n(u)=\nu+\int_{(0,t]}\kappa(u-s)\,n(\mathrm ds),\quad u>t.
\end{equation}
Here $N^{t,n}$ counts only events after $t$. This law is defined for
every finite history by taking Poisson roots with deterministic rate
$b_n$ after $t$ and attaching independent families. It is a measurable
continuation kernel; in particular, it does not require conditioning on
an event at a  deterministic time.

For bounded $f:\mathbb N_0\to\mathbb R$, put
\begin{equation}\label{eq:hawkes-value}
 U_f(t,n)=\mathbb E^{t,n}
       f\bigl(n((0,t])+N^{t,n}((t,T])\bigr).
\end{equation}
Write $n_t=\mathrm dN|_{(0,t]}$ and
$n_{t-}=\mathrm dN|_{(0,t)}$. Independent future offspring increments
in the cluster construction show that
$Y_t=U_f(t,n_t)=\mathbb E[f(N_T)\mid\mathcal F_t]$.
The current-time difference is
\begin{equation}\label{eq:hawkes-derivative}
 g_f(t)=U_f(t,n_{t-}+\delta_t)-U_f(t,n_{t-}).
\end{equation}
It is a predictable field, and
\begin{equation}\label{eq:hawkes-representation}
 \nabla_J^{\mathbb P}Y(t,1)=g_f(t),\qquad
 f(N_T)=\mathbb E f(N_T)
           +\int_0^T g_f(t)(\mathrm dN_t-\lambda_t\,\mathrm dt).
\end{equation}
In particular,
\begin{equation}\label{eq:hawkes-energy}
 \mathbb E\int_0^T |g_f(t)|^2\lambda_t\,\mathrm dt
       =\operatorname{Var}(f(N_T)).
\end{equation}

To verify the functional identification, first fix a history $n$
supported in $(0,s]$, with $s<t$. Couple its continuations from $s$
and $t$ by using the same roots after $t$ and the same families.
They differ only if a root occurs in $(s,t]$. Hence
\[
 |U_f(s,n)-U_f(t,n)|
   \leq 2\|f\|_\infty\int_s^t b_n(u)\,\mathrm du.
\]
Since the realised history is finite on $[0,T]$, this proves
c\`adl\`ag compatibility, and at every event time
\[
 Y_{t-}=U_f(t,n_{t-}),\qquad \Delta Y_t=g_f(t).
\]
The history-valued process $t\mapsto n_{t-}$ is left-continuous
and adapted, hence predictable; measurability of the continuation
kernel therefore makes \eqref{eq:hawkes-derivative} predictable. Moreover,
$|g_f|\leq2\|f\|_\infty$ and \eqref{eq:hawkes-moments} imply finite
predictable energy. The representation property for a nonexplosive
point process in its natural filtration \cite{Jacod75}, together with
Proposition~\ref{prop:identification}, proves
\eqref{eq:hawkes-representation}; the isometry gives
\eqref{eq:hawkes-energy}.

\paragraph{The effect on future intensity.}
For $n$ supported in $(0,t)$, let
$Z^{t,n}=n((0,t))+N^{t,n}((t,T])$ under the unperturbed continuation.
Let $C_u$ be the number of events up to age $u$ in an independent
family rooted at age zero, including the root itself. Adding $\delta_t$
to the history adds the direct excitation $\kappa(u-t)$ to
\eqref{eq:hawkes-continuation} for every $u>t$. By superposition of
independent Poisson offspring processes, the two continuations can
be coupled so that their terminal counts are $Z^{t,n}$ and
$Z^{t,n}+C_{T-t}$. Therefore
\begin{equation}\label{eq:hawkes-cluster-difference}
 U_f(t,n+\delta_t)-U_f(t,n)
  =\mathbb E\left[f\bigl(Z^{t,n}+C_{T-t}\bigr)-f\bigl(Z^{t,n}\bigr)\right],
\end{equation}
where $C$ is independent of the unperturbed continuation.
The stochastic derivative thus includes the inserted event and its
future descendants. In the functional insertion, the observed
compensator history remains fixed, as in Section~\ref{sec:functionals}; the changed
future dynamics are encoded in the conditional-value functional
$U_f$. No compensator value before or at the insertion time is changed.

An explicit bounded example is $f(n)=z^n$, $0<z<1$. Set
$A_z(u)=\mathbb E[z^{C_u}]$. The branching construction gives
\begin{equation}\label{eq:hawkes-cluster-pgf}
 A_z(u)=z\exp\left\{\int_0^u
             \kappa(s)\bigl(A_z(u-s)-1\bigr)\,\mathrm ds\right\}.
\end{equation}
The Poisson generating functional then yields
\begin{align}
 U_z(t,n)
   &=z^{n((0,t])}\exp\left\{\int_t^T
                b_n(u)\bigl(A_z(T-u)-1\bigr)\,\mathrm du\right\},
                 \label{eq:hawkes-pgf-value}\\
 g_z(t)&=U_z(t,n_{t-})\bigl(A_z(T-t)-1\bigr).
                 \label{eq:hawkes-pgf-derivative}
\end{align}
This integrand depends on the past events through the entire
future excitation profile $b_{n_{t-}}$.
Finally, let
\[
 R(u)=\mathbb E C_u
     =1+\int_0^u\sum_{k\geq1}\kappa^{*k}(s)\,\mathrm ds
     \leq\frac{1}{1-\|\kappa\|_1},
\]
where $\kappa^{*k}$ denotes $k$-fold convolution on $[0,\infty)$.
If $|f(n+1)-f(n)|\leq L$, then
\eqref{eq:hawkes-cluster-difference} gives
$|g_f(t)|\leq L R(T-t)$. Applying this bound to
$f_M(n)=n\wedge M$, and using \eqref{eq:hawkes-moments}, dominated
convergence and the isometry give the explicit identity
\begin{equation}\label{eq:hawkes-terminal-count}
 N_T-\mathbb E N_T
       =\int_0^T R(T-t)(\mathrm dN_t-\lambda_t\,\mathrm dt).
\end{equation}
The coefficient $R(T-t)$ is the expected total contribution, up to
$T$, of an event inserted at time $t$.
\section{Relation with Malliavin calculi on Wiener-Poisson space}\label{sec:malliavin}
The stochastic derivative $\nabla^{\mathbb{P}}_{X,J}$ is defined for any continuous martingale $X$ and integer-valued random measure $J$ satisfying Assumption \ref{ass:main}. In the case where $X=W$ is the Wiener process and $J$ a Poisson random measure, we can recover relations with the Malliavin calculus on the Wiener-Poisson space.

In the Wiener case, the vertical derivative $\nabla^{\mathbb{P}}_W$ was shown in \cite{CF13} to be the predictable projection of the Malliavin derivative on the Wiener space.

In the case of jump processes,
different constructions of Malliavin-type derivative operators have been proposed on the Poisson and Wiener-Poisson space \cite{BGJ87,Bismut83,BassCranston86,CarlenPardoux1990,IK06,Ishikawa16,KS75}. These constructions lead to distinct {\it anticipative} operators. 

We discuss in this section the relation between the stochastic derivative and various notions of anticipative derivatives defined on the Poisson space.
The core of this relation is determined by the divergence on predictable integrands.
We first state this observation in a form which also fixes the
spaces in the projection diagram. We then apply it to
Wiener--Poisson differentiation and to the normalized jump
derivative.
\begin{proposition}[Duality criterion]\label{prop:duality-criterion}
Let $D:\Dom D\subset L^2(\mathcal F_T)\to\mathcal G$ be a
densely defined closed operator into a Hilbert space $\mathcal G$,
with adjoint $\delta$. Suppose that there is a linear isometry
$T:\HXJ\to\mathcal G$ such that
\begin{equation}\label{eq:divergence-restriction}
 T(\HXJ)\subset\Dom\delta,\qquad
 \delta(Th)=(I_{X,J}h)_T\quad(h\in\HXJ).
\end{equation}
Then, for every $\xi\in\Dom D$,
\begin{equation}\label{eq:abstract-projection}
 \mathsf D_{X,J}\xi=T^*D\xi,
 \qquad T\mathsf D_{X,J}\xi=P_{T(\HXJ)}D\xi.
\end{equation}
\end{proposition}

\begin{proof}
For $h\in\HXJ$, duality and \eqref{eq:divergence-restriction}
give
\[
 \E[\xi(I_{X,J}h)_T]
   =\langle D\xi,Th\rangle_{\mathcal G}
   =\langle T^*D\xi,h\rangle_{\HXJ}.
\]
Equation~\eqref{eq:adjoint-terminal} identifies $T^*D\xi$.
Since $T$ is an isometry, $TT^*$ is the orthogonal projection onto
its closed range. This proves the second identity.
\end{proof}
For an anticipative derivative whose divergence satisfies
the restriction~\eqref{eq:divergence-restriction}, Proposition~\ref{prop:duality-criterion} identifies the
stochastic derivative with $T^*D$ on $\operatorname{Dom}D$.
When $T$ is the inclusion of predictable fields, this
identity is a predictable-projection formula.
Local mark gradients require a different comparison,
developed in Sections~\ref{sec.localmark} and \ref{subsec:bismut}.
\subsection{A lifting theorem on Wiener--Poisson space}\label{sec.WienerPoisson}

Let $W$ be a $d$-dimensional Brownian motion and $\eta$ an
independent Poisson random measure with deterministic
$\sigma$-finite intensity $\lambda$ on $(0,T]\times E$,
where $\lambda(\{t\}\times E)=0$. Use their usual natural
filtration, with trivial initial $\sigma$-field. Define
\begin{align}\label{eq:anticipative-space}
 \mathcal G_{W,\eta}
 =L^2(\Pp\otimes\mathrm dt;\R^d)
          \oplus L^2(\Pp\otimes\lambda),&&
 {\mathcal H}_{W,\eta}=\mathcal L^2_{\Pp}([W])\oplus\mathcal L^2_{\Pp}(\lambda).
\end{align}
Let $\Pi_{\mathrm{pred}}$ be the orthogonal projection from
$\mathcal G_{W,\eta}$ onto $\mathcal H_{W,\eta}$.
The Poisson add-one derivative of $\xi=f(W,\eta)$ is
\begin{equation}\label{eq:add-one}
 \dadd_{t,z}\xi=f(W,\eta+\delta_{(t,z)})-f(W,\eta).
\end{equation}
It is an operator on equivalence classes for
$\Pp\otimes\lambda$, with domain specified by square
integrability of the difference. In the pure Poisson case, writing
$\xi=\E\xi+\sum_{n\geq1}I_n(f_n)$, this domain is
\begin{equation}\label{eq:chaos-domain}
 \mathbb D^{1,2}_\eta
  =\left\{\xi\in L^2:\sum_{n\geq1}n\,n!
                 \norm{f_n}_{L^2(\lambda^n)}^2<\infty\right\},
 \qquad
 \dadd_{t,z}\xi=\sum_{n\geq1}nI_{n-1}(f_n((t,z),\cdot)).
\end{equation}
Thus the chaos derivative and the Picard add-one derivative agree
on their $L^2$ domain; see \cite{Last14,Ishikawa16}.
On the product space let
$D=(D^W,\dadd)$, with domain $\mathbb D^{1,2}_{W,\eta}$,
the intersection of the Wiener and Poisson derivative domains.

\begin{theorem}[Predictable projection and commutative diagrams]
\label{thm:lifting}
For $\xi\in\mathbb D^{1,2}_{W,\eta}$,
\begin{equation}\label{eq:projection-theorem}
 \nabla^{\Pp}_{W,\eta}\CE\xi
       =\Pi_{\mathrm{pred}}(D^W\xi,\dadd\xi).
\end{equation}
Equivalently, the following diagram commutes in the stated $L^2$
spaces; either noise may be omitted:
\begin{equation}\label{eq:lifting-diagram}
\begin{CD}
 \mathbb D^{1,2}_{W,\eta} @>{(D^W,\dadd)}>> \mathcal G_{W,\eta}\\
 @V{\CE}VV @VV{\Pi_{\mathrm{pred}}}V\\
 \mathcal I^2_{\Pp}([W],\lambda)
       @>>{\nabla^{\Pp}_{W,\eta}}> \mathcal H_{W,\eta}
\end{CD}
\end{equation}
The upper map followed by predictable projection has a unique
continuous extension to all $L^2(\mathcal F_T)$, and
\begin{equation}\label{eq:all-L2}
 \norm{\nabla^{\Pp}_{W,\eta}\CE\xi}_{\mathcal H_{W,\eta}}^2
       =\Var(\xi).
\end{equation}

In the pure Poisson case, if $E\subset\R\setminus\{0\}$ and
$\widehat D_{t,z}\xi=z^{-1}\dadd_{t,z}\xi$, then
\begin{equation}\label{eq:quotient-projection}   \bigl(\nabla^{\Pp}_\eta\mathsf M\xi\bigr)(t,z)
    = z\,\bigl(\Pi^{(z^2\lambda)}_{\mathrm{pred}}\widehat D\xi\bigr)(t,z).
    \end{equation}
The corresponding diagram is
\begin{equation}\label{eq:quotient-diagram}
\begin{CD}
 \mathbb D^{1,2}_\eta @>{\widehat D}>> L^2(\Pp\otimes z^2\lambda)\\
 @V{\CE}VV @VV{\times z\ \Pi_{\mathrm{pred}}^{(z^2\lambda)}}V\\
 \mathcal I^2_{\Pp}(\lambda)
       @>>{\nabla^{\Pp}_\eta}> {\mathcal L^2_{\Pp}(\lambda)} .
\end{CD}
\end{equation}
\end{theorem}

\begin{proof}
The Wiener and Poisson divergences agree with their stochastic
integrals on predictable square-integrable fields. Apply
Proposition~\ref{prop:duality-criterion} with $T$ the inclusion
of $\mathcal H_{W,\eta}$ into $\mathcal G_{W,\eta}$.
Its adjoint is $\Pi_{\mathrm{pred}}$, so this first identifies
the stochastic derivative.

The pair has the predictable representation property. For
completeness, bounded variables $\xi_W$ and $\xi_\eta$ depending
on the two separate noises give conditional martingales
$U_t=\E[\xi_W\mid\mathcal F_t^W]$ and
$V_t=\E[\xi_\eta\mid\mathcal F_t^\eta]$.
Each has a representation in its own filtration, by the Brownian
and Poisson representation theorems \cite{ikeda2014}.
Independence gives
$\E[\xi_W\xi_\eta\mid\mathcal F_t]=U_tV_t$.
The product formula, with zero continuous--discontinuous
covariation, represents $UV-U_0V_0$ using $(W,\eta-\lambda)$.
The bounded factors make its integrands square integrable.
Finite sums of these products are dense in the product $L^2$
space. Closedness of the integral range proves the property and
identifies the adjoint with the stochastic derivative. This proves
\eqref{eq:projection-theorem} and \eqref{eq:lifting-diagram}.
Isometry gives \eqref{eq:all-L2}. If
$\xi_n\in\mathbb D^{1,2}_{W,\eta}$ converges to $\xi$ in $L^2$,
then
\[
 \norm{\Pi_{\mathrm{pred}}D\xi_n-
              \Pi_{\mathrm{pred}}D\xi_m}_{\mathcal H_{W,\eta}}^2
       =\Var(\xi_n-\xi_m).
\]
Density and completeness prove the unique continuous extension.
No extension of the unprojected derivative to all $L^2$ is asserted.

Multiplication by $z$ is unitary from $L^2(\Pp\otimes z^2\lambda)$ onto $L^2(\Pp\otimes\lambda)$ and preserves predictability, hence commutes with predictable projection. 
Since  $z\widehat D=D^{\mathrm{add}}$,  the pure Poisson case of
\eqref{eq:projection-theorem} proves
\eqref{eq:quotient-projection}--\eqref{eq:quotient-diagram}.
\end{proof}
Theorem~\ref{thm:lifting} is the Wiener--Poisson counterpart of
\cite[Theorem 6.1]{CF13}. 
The pair $(D^W,D^{\mathrm{add}})$ is the derivative used in the
Malliavin calculus on Wiener--Poisson space of Ishikawa and Kunita
\cite{IK06,Ishikawa16}, so Theorem~\ref{thm:lifting} identifies the stochastic
derivative as the predictable projection of their derivative.
For the Poisson chaos, add-one, and
quotient conventions, it gives the precise identification on their
common domains \cite{Ishikawa16,Last14,SU16}.
In \eqref{eq:quotient-projection}, the normalized coefficient is
integrated against $z(\eta-\lambda)$ and has energy measure
$z^2\lambda$. There is no scalar division-by-mark convention for
a general $E$. The all-$L^2$ statement is consistent with the
Poisson representation in \cite{LP11}; it does not require Malliavin differentiability
of the terminal variable.

\paragraph{The Kabanov--Skorokhod calculus.}
Kabanov and Skorokhod~\cite[Sections~2--3]{KS75}
construct a  differential operator and an extended stochastic
integral from the chaos of a centered stochastic measure with independent
values. In the Poisson setting above, their derivative
$D^{\mathrm{KS}}$ is the chaos derivative in
\eqref{eq:chaos-domain}; its domain is $\mathbb D^{1,2}_\eta$,
and it equals $\dadd$ on that domain. For
$u=\sum_{n\geq0}I_n(u_n)\in L^2(\Pp\otimes\lambda)$, their
extended integral is
\[
 \delta^{\mathrm{KS}}u
    =\sum_{n\geq0}I_{n+1}(\widetilde u_n),\qquad
 \Dom\delta^{\mathrm{KS}}
    =\left\{u:\sum_{n\geq0}(n+1)!
           \norm{\widetilde u_n}_{L^2(\lambda^{n+1})}^2<\infty\right\},
\]
Here $u_n$ is symmetric in its $n$ chaos variables, and
$\widetilde u_n$ is its symmetrization in all $n+1$ variables.
The chaos isometry gives
$\delta^{\mathrm{KS}}=(D^{\mathrm{KS}})^*$.
For predictable $u\in L^2(\Pp\otimes\lambda)$, ordering the time
variables in the chaos kernels yields
$\delta^{\mathrm{KS}}u=(u*(\eta-\lambda))_T$; for the product
intensity $\lambda(\mathrm dt,\mathrm dz)=\mathrm dt\,\nu(\mathrm dz)$,
this is the iterated-integral construction in
\cite[Theorem~7.2]{KS75}. Thus
Proposition~\ref{prop:duality-criterion} gives
\[
 \nabla_\eta^{\Pp}\CE\xi
       =\Pi_{\mathrm{pred}}D^{\mathrm{KS}}\xi,
 \qquad \xi\in\mathbb D^{1,2}_\eta.
\]
This is the pure Poisson instance of \eqref{eq:lifting-diagram}.
For this product intensity, the predictable representation property
is already stated in
\cite[Theorem~7.3]{KS75}. These chaos constructions
use a deterministic control measure and do not by themselves define
an anticipative derivative for a general random compensator.

\begin{remark}[Larger filtrations]
If the filtration is larger than the one generated by $W$ and $\eta$, provided the divergence
restriction~\eqref{eq:divergence-restriction}  remains valid, diagram~\eqref{eq:lifting-diagram} remains
commutative after replacing its lower-left space
$\mathcal I^2_{\mathbb{P}}([W],\lambda)$ by $\mathcal M_0^2$ and its lower arrow
by the stochastic derivative defined in~\eqref{eq:weak-definition}.
The
representation of $\xi$ by this derivative then holds only if the
 martingale $\CE\xi$ belongs to the integral range.\end{remark}

\subsection{Local mark gradients}\label{sec.localmark}
Several Malliavin-type calculi on Poisson space differentiate a functional
with respect to the positions of the marks already present
\cite{BGJ87,BassCranston86,BD09}. In general such a gradient does not determine the stochastic
derivative. Indeed, let $0<\nu(E)<\infty$ and $U=\eta((0,T]\times E)$.
Displacing marks does not change $U$, so every gradient $G$ of this kind
satisfies $GU=0$ whenever $U\in\operatorname{Dom}G$, whereas $\CE U_t=\eta((0,t]\times E)-t\nu(E)$
gives $\nabla^{\Pp}_\eta\CE U=1$. Hence there is no linear map $R$
with $\nabla^{\Pp}_\eta\CE F=R(GF)$ on a domain containing $U$. We
illustrate this for the lent particle method of Bouleau and Denis \cite{BD09}, where
the gradient factors through the add-one difference.

Take $\lambda(\mathrm dt,\mathrm dz)=\mathrm dt\,\nu(\mathrm dz)$
and equip the mark space with a local symmetric Dirichlet form
\[
 e(f,g)=\frac12\int_E\gamma[f,g]\,\mathrm d\nu,
 \qquad f,g\in\mathbf d,
\]
and a bottom gradient $b$ with values in $L^2(R,\rho)$, normalized
by $\int f^b\,\mathrm d\rho=0$ and
$\int f^b g^b\,\mathrm d\rho=\gamma[f,g]$.
Assume $1\in\mathbf d_{\mathrm{loc}}$ and $\gamma[1]=0$, together
with the hypotheses for the associated Poisson Dirichlet structure \cite{BD09}. Time carries the trivial
differential structure. Write $\varepsilon^+$ for insertion and
$\varepsilon^-$ for removal, and let $\eta\odot\rho$ denote
the independently marked Poisson configuration. The upper gradient is
\begin{equation}\label{eq:lent}
 U^\sharp=\mathcal A\,b(\varepsilon^+U),\qquad
 \mathcal A H=\int\varepsilon^-H(t,z,r)
                    (\eta\odot\rho)(\mathrm dt,\mathrm dz,\mathrm dr).
\end{equation}
Let $\widehat\E$ denote expectation over the auxiliary marks and
$\Gamma[U,V]=\widehat\E[U^\sharp V^\sharp]$.

\begin{proposition}[Local factorization and obstruction to projection]
\label{prop:local}
On a common cylindrical core on which the indicated energies are
finite,
\begin{equation}\label{eq:local-factorization}
 U^\sharp=\mathcal A\,b(\dadd U),\qquad
 \E\Gamma[U,V]
   =\E\int_{(0,T]\times E}
          \gamma_z[\dadd U,\dadd V]\,\lambda(\mathrm dt,\mathrm dz).
\end{equation}
In general there is no linear map $R$ from the upper gradient
space into ${\mathcal L^2_{\Pp}(\lambda)}$ such that
\begin{equation}\label{eq:no-local-projection}
 \nabla^{\Pp}_\eta\CE U=R(U^\sharp)
\end{equation}
for every $U$ in the common derivative domain. In particular the
projection square \eqref{eq:lifting-diagram} cannot in general
be formed with $\sharp$ in place of $\dadd$.
\end{proposition}

\begin{proof}
At fixed configuration, $U(\eta)$ is constant in the added mark.
Since $b$ annihilates constants,
$b(\varepsilon^+U)=b(\varepsilon^+U-U)=b(\dadd U)$.
Substitution in \eqref{eq:lent} gives the first identity.
The lent-particle energy formula and the Mecke identity give
\[
 \E\Gamma[U,V]
 =\E\int\varepsilon^-\gamma_z[\varepsilon^+U,\varepsilon^+V]
                                                  \,\mathrm d\eta
 =\E\int\gamma_z[\dadd U,\dadd V]\,\mathrm d\lambda.
\]
The diagonal identity follows first for nonnegative energies;
polarization gives the bilinear statement.

For the second assertion take $0<\nu(E)<\infty$ and a
conservative bottom structure with $1$ in its domain and
$\gamma[1]=0$. The variable $U=\eta((0,T]\times E)$ belongs
to both derivative domains and satisfies
\[
 U^\sharp=0,\qquad \dadd_{t,z}U=1,\qquad
 \nabla^{\Pp}_\eta\CE U=1.
\]
The first identity follows from \eqref{eq:lent}, or by approximating
the count by smooth functions of the count. The last follows from
$\CE U=\eta((0,t]\times E)-t\nu(E)$.
A linear $R$ would send zero to the nonzero predictable field $1$,
a contradiction.
\end{proof}
The factorization of a local gradient through the add-one difference is
noted, for the intrinsic gradient on configuration space, in
\cite[Sec.~2]{IK06}, following \cite{Pri98}. 
\begin{remark}[Differentiation with respect to jump times]\label{rem:jump-time}
The gradients considered above differentiate in the mark only. A third
family of gradients on Poisson space, introduced by Carlen and Pardoux
\cite{CarlenPardoux1990} for the standard Poisson process, differentiates with respect
to shifts of the jump times. On a bounded time interval, the adjoint of
their gradient coincides with the compensated stochastic integral only on
integrands whose time integral vanishes \cite[Theorem~3.1]{CarlenPardoux1990}, so it
does not satisfy a Clark-type formula. The obstruction is the one in
Proposition~\ref{prop:local}: the count
$U=\eta((0,T]\times E)$ is invariant under time changes of $(0,T]$ that
fix the endpoints, so its derivatives in these directions vanish, whereas
its stochastic derivative is the constant field $1$. Thus this gradient,
like the mark gradients above, does not determine the stochastic
derivative.

Privault \cite{Pri94} proposed, for the standard Poisson process on the
half-line, a modification of this gradient whose adjoint coincides with
the compensated stochastic integral on square-integrable predictable
integrands. 
For terminal variables in the square-integrable Sobolev
domain of the modified gradient, the duality argument
of Proposition~\ref{prop:duality-criterion}  identifies its predictable projection
with the stochastic derivative.
\end{remark}
\subsection{Relation with Bismut's perturbation approach}
\label{subsec:bismut}

Bismut's construction differentiates a quasi-invariance identity
for perturbations of the driving paths. In the jump setting,
smooth displacement of existing marks leads to an integration-by-parts
formula. His treatment already includes state-dependent changes
of intensity by stochastic exponentials and a relation with
martingale representation \cite[Sections~1--3]{Bismut83}.
We state the comparison on a common Poisson core, distinguishing
the derivative in the mark from the stochastic derivative.

Let $E\subset\R^d$ be open, and let $J$ be a Poisson random measure
with intensity
\[
 \lambda(\mathrm dt,\mathrm dx)=\mathrm dt\,\rho(x)\,\mathrm dx,
 \qquad \rho\in C^2(E),\quad \rho>0.
\]
Use the usual natural filtration of $J$ and a finite horizon $T$.
The intensity is finite on compact subsets; its total mass may be
infinite. Let $\mathcal C$ be the class of cylinders
\begin{equation}\label{eq:bismut-core}
 F=\varphi(J(f_1),\ldots,J(f_m)),\qquad
 \varphi\in C_b^2(\R^m),\quad
 f_j\in C_c^2((0,T)\times E).
\end{equation}
Here $J(f)=\int f\,\mathrm dJ$. For every $F\in L^2(\Pp)$ define
\[
 k_F=\nabla_J^{\Pp}\CE F.
\]
For $F\in\mathcal C$ also put
\[
 G_F(t,x)=\nabla_x\dadd_{t,x}F,
\]
where $\nabla_x$ is ordinary differentiation of the inserted mark,
with the original configuration held fixed.

\begin{proposition}[Mark displacement, score, and differentiated projection]
\label{prop:bismut-comparison}
Let $v\in C_c^3((0,T)\times E;\R^d)$, and let
$\theta_\varepsilon(t,\cdot)$ be its mark flow:
\[
 \partial_\varepsilon\theta_\varepsilon(t,x)
       =v(t,\theta_\varepsilon(t,x)),\qquad
 \theta_0(t,x)=x.
\]
Write $\Theta_\varepsilon J$ for the configuration obtained by
replacing each atom $(t,x)$ by $(t,\theta_\varepsilon(t,x))$,
and define
\begin{equation}\label{eq:bismut-score}
 h_v(t,x)=-\frac{\operatorname{div}_x(\rho(x)v(t,x))}{\rho(x)}.
\end{equation}
Then the law of $\Theta_\varepsilon J$ is equivalent to the law
of $J$. Its likelihood $\Lambda_\varepsilon$ satisfies
\begin{equation}\label{eq:bismut-likelihood-derivative}
 \frac{\Lambda_\varepsilon-1}{\varepsilon}
       \longrightarrow (h_v*(J-\lambda))_T
       \quad\text{in }L^2(\Pp),\qquad\varepsilon\longrightarrow0.
\end{equation}
Consequently, for every $F\in L^2(\Pp)$,
\begin{equation}\label{eq:bismut-law-response}
 \left.\frac{\mathrm d}{\mathrm d\varepsilon}
             \E[F(\Theta_\varepsilon J)]\right|_{\varepsilon=0}
       =\E\int k_F h_v\,\mathrm d\lambda.
\end{equation}
For $F\in\mathcal C$, the  directional derivative
\[
 D_v^{\mathrm B}F
       =\left.\frac{\mathrm d}{\mathrm d\varepsilon}
                     F(\Theta_\varepsilon J)\right|_{\varepsilon=0}
\]
exists in $L^2$, and the more precise identities are
\begin{equation}\label{eq:bismut-pairing}
 \E[D_v^{\mathrm B}F]
   =\E\int k_F h_v\,\mathrm d\lambda
   =\E\int v\cdot\nabla_x k_F\,\mathrm d\lambda,
 \qquad
 \nabla_x k_F=\Pi_{\mathrm{pred}}G_F.
\end{equation}
The mark derivative of $k_F$ can be taken in the weak sense; it
belongs to $\mathcal L^2_{\Pp}(\lambda;\mathbb R^d)$.
In particular,
\begin{equation}\label{eq:bismut-pairing-bound}
 |\E[D_v^{\mathrm B}F]|
\leq\sqrt{\Var(F)}\,\norm{h_v}_{L^2(\lambda)}.
\end{equation}
The operator $G:F\mapsto\nabla_x\dadd F$ on $\mathcal C$ is
closable from $L^2(\Pp)$ to $L^2(\Pp\otimes\lambda;\R^d)$.
Writing $\overline G$ for its closure, the projection statement
extends to every $F\in\Dom\overline G$ and gives the commutative
diagram
\begin{equation}\label{eq:bismut-diagram}
\begin{CD}
 \Dom\overline G @>{\overline G}>>
                     L^2(\Pp\otimes\lambda;\R^d)\\
 @V{\CE}VV @VV{\Pi_{\mathrm{pred}}}V\\
 \CE(\Dom\overline G) @>>{\nabla_x\nabla_J^{\Pp}}>
                     \mathcal L^2_{\Pp}(\lambda;\mathbb R^d).
\end{CD}
\end{equation}
Thus its lower arrow includes a mark derivative of the 
stochastic derivative.
\end{proposition}

\begin{proof}
The flow is a diffeomorphism of $E$ and is the identity outside
a fixed compact set. The mapping theorem for Poisson measures gives
the intensity $r_\varepsilon\lambda$ for $\Theta_\varepsilon J$,
where
\[
 r_\varepsilon(t,x)
 =\frac{\rho(\theta_{-\varepsilon}(t,x))}{\rho(x)}
       \left|\det D_x\theta_{-\varepsilon}(t,x)\right|.
\]
For small $|\varepsilon|$, this ratio is bounded above and away
from zero, equals one off a fixed set $A$ of finite $\lambda$-mass,
and satisfies uniformly on $A$
\begin{equation}\label{eq:bismut-ratio-expansion}
 r_\varepsilon=1+\varepsilon h_v+O(\varepsilon^2).
\end{equation}
The sign in \eqref{eq:bismut-score} follows by differentiating
both the inverse flow and its Jacobian. The Poisson likelihood is
\begin{equation}\label{eq:bismut-likelihood}
 \Lambda_\varepsilon
  =\exp\left\{\int\log r_\varepsilon\,\mathrm dJ
                   -\int(r_\varepsilon-1)\,\mathrm d\lambda\right\}.
\end{equation}
It is strictly positive and has mean one. Smoothness on the compact
set $A$ gives uniform bounds on the first two parameter derivatives
of $r_\varepsilon$ and $\log r_\varepsilon$. Differentiating
\eqref{eq:bismut-likelihood} at zero therefore gives
\[
 \left.\partial_\varepsilon\Lambda_\varepsilon\right|_0
       =\int h_v\,\mathrm d(J-\lambda).
\]
Taylor's formula bounds the difference quotient error by
$C|\varepsilon|(1+J(A))^2\exp\{C\varepsilon_0J(A)+C\}$
for $|\varepsilon|\leq\varepsilon_0$ sufficiently small.
The Poisson variable $J(A)$ has exponential moments of every
order, proving \eqref{eq:bismut-likelihood-derivative}.
The same reasoning, or a larger compact set, gives equivalence
for each fixed parameter for which the flow is considered.

For $F\in L^2$, the law identity
$\E[F(\Theta_\varepsilon J)]=\E[F\Lambda_\varepsilon]$
and $L^2$ differentiation yield
\[
 \left.\partial_\varepsilon\E[F(\Theta_\varepsilon J)]\right|_0
 =\E[F(h_v*(J-\lambda))_T]
 =\E\int k_Fh_v\,\mathrm d\lambda,
\]
where the last equality is the duality \eqref{eq:adjoint-terminal}.

Now let $F$ have the form \eqref{eq:bismut-core} and set
$a=(J(f_1),\ldots,J(f_m))$. Direct differentiation gives
\begin{align*}
 D_v^{\mathrm B}F
   =\sum_{j=1}^m\partial_j\varphi(a)
                         J(v\cdot\nabla_x f_j),&&
 G_F(t,x)
=\sum_{j=1}^m\partial_j\varphi(a+f(t,x))\nabla_x f_j(t,x).
\end{align*}
The difference quotients defining $D_v^{\mathrm B}F$ are bounded
by a constant times the count in a fixed finite-intensity set.
Dominated convergence thus holds in $L^2$. This proves the first
equality in \eqref{eq:bismut-pairing}.

The predictable-projection theorem gives
$k_F=\Pi_{\mathrm{pred}}\dadd F$. Since $G_F$ is bounded
and supported on a fixed finite-intensity set, mark differentiation
commutes with this conditional expectation. Equivalently, test
against compactly supported smooth mark functions and use Fubini
to obtain
$\nabla_x k_F=\Pi_{\mathrm{pred}}G_F$ in the weak sense.
This proves the projection identity on $\mathcal C$.
Weighted integration by parts in the
mark, with no boundary term because $v$ is compactly supported,
then gives
\[
 \E\int k_Fh_v\,\mathrm d\lambda
 =-\E\int_0^T\!\int_E
       k_F\operatorname{div}_x(\rho v)\,\mathrm dx\,\mathrm dt
 =\E\int v\cdot\nabla_x k_F\,\mathrm d\lambda.
\]
Finally, Cauchy--Schwarz and the stochastic-integral isometry give
\eqref{eq:bismut-pairing-bound}.

To prove closability, let $a(J)$ be a bounded measurable variable
and $w\in C_c^1((0,T)\times E;\R^d)$. Set
$b_w=\rho^{-1}\operatorname{div}_x(\rho w)$. Mark integration
by parts and the Mecke formula give, for $F\in\mathcal C$,
\[
 \E\int a(J)G_F\cdot w\,\mathrm d\lambda
   =-\E[F\,\delta(a b_w)],
\]
where
\[
 \delta(a b_w)
   =\int a(J-\delta_{(t,x)})b_w(t,x)\,J(\mathrm dt,\mathrm dx)
                 -a(J)\int b_w\,\mathrm d\lambda.
\]
This variable is in $L^2$: its absolute value is bounded by a
constant times one plus the count in the compact support of $w$.
If $F_n\to0$ in $L^2$ and $G_{F_n}\to H$ in
$L^2(\Pp\otimes\lambda;\R^d)$, these identities give
$\E\int aH\cdot w\,\mathrm d\lambda=0$ for every such $a,w$.
These tensor test fields are dense, so $H=0$.
This proves closability. If instead $F_n\to F$ and
$G_{F_n}\to\overline GF$, continuity of the stochastic
derivative gives $k_{F_n}\to k_F$ in $L^2(\Pp\otimes\lambda)$.
Closedness of the weighted weak mark derivative then implies
\[
 \nabla_x k_F=\Pi_{\mathrm{pred}}\overline GF.
\]
This proves \eqref{eq:bismut-diagram} on the stated domain.
No condition on $\dadd F$ is needed for a general
$F\in\Dom\overline G$.
\end{proof}

The field $h_v$ is bounded and has finite deterministic quadratic
energy. The response in \eqref{eq:bismut-law-response} is therefore
also the first-order response to the multiplicative perturbation
$(1+\varepsilon h_v)\lambda$ studied in the compensator section.
The transported intensity $r_\varepsilon\lambda$ and this linear
perturbation have the same tangent at zero; their finite-parameter
laws need not coincide. Equation~\eqref{eq:bismut-law-response}
extends to every $L^2$ variable, whereas the directional derivative
$D_v^{\mathrm B}F$ requires its own domain.

\begin{remark}
Bismut's \cite[Theorem~3.5]{Bismut83} identifies the mark derivative of a representing
integrand with a field obtained from an optional projection.
Equation~\eqref{eq:bismut-diagram} gives the homogeneous Poisson
comparison in the present notation. For state-dependent jump
kernels his integration-by-parts formula also contains terms
arising from the propagated state variation
\cite[Theorem~2.5]{Bismut83}. Thus substituting a random density
for $\rho$ in \eqref{eq:bismut-score} does not give the general
score: perturbing the past can change future compensators.
More generally, let $F\in L^2(\Pp)$ and
$a\in\HJ$. Whenever an admissible path
perturbation has $L^2$ likelihood derivative $(a*(J-\mu))_T$, duality gives
\[
 \E[D^{\mathrm B}F]=\E\int(\nabla_J^{\mathbb P}\CE F)a\,\mathrm d\mu
\]
if the directional difference quotients also converge in
$L^1$. This statement needs no representation property.
\end{remark}

\begin{remark}[Necessity of the extra mark derivative]\label{rem:extra-mark}
For a concrete example take $E=\R^d$, a smooth strictly positive
integrable density $\rho$, and
$F=J((0,T]\times E)$. Every mark flow above preserves $F$, so
$D_v^{\mathrm B}F=0$ for every $v$. Also
$\overline GF=0$, whereas $\nabla_J^{\Pp}\CE F=1$.
Indeed, approximate the count by smooth compactly supported
mark cutoffs tending to one with gradients bounded by $C/R$,
and use smooth truncations of the resulting counts in
\eqref{eq:bismut-core}. Their mark energies tend to zero and
their terminal values converge to $F$ in $L^2$; time cutoffs
near $0$ and $T$ have no mark-energy cost in the limit.
Consequently no linear projection of the family of mark-directional
derivatives can recover the stochastic derivative on their
full common domain. For compactly supported mark flows,
$\int h_v(t,x)\rho(x)\,\mathrm dx=0$; these directions preserve
total intensity at each time. General compensator perturbations
can change it. This explains both the differentiated lower arrow
in \eqref{eq:bismut-diagram} and the difference between the two
classes of perturbations.
\end{remark}

\section{Perturbations of predictable compensators}
\label{sec:compensator-perturbations}

We apply the stochastic derivative to multiplicative perturbations
of a predictable compensator. A deterministic bound on the accumulated
quadratic energy of the perturbation gives a series expansion of the
likelihood in the space of square-integrable martingales. Pairing this
expansion with the  stochastic derivative yields response formulas for
every square-integrable terminal variable. Neither a predictable
representation property nor Malliavin differentiability of the terminal
variable is required.

Fix $T<\infty$. As before, we work under Assumption \ref{ass:main}: outside a $\Pp$-null set,
\begin{equation}\label{eq:perturb-standing}
\forall t\in (0,T],\qquad J(\{t\}\times E)\leq1,\qquad
 \mu(\{t\}\times E)=0\quad
\end{equation}
and the predictable control measure $m$ defined in \eqref{eq:control} is 
$\sigma$-finite. Recall the notation $
 I_Jv=v*(J-\mu) $
where $*$ denotes integration on $(0,t]\times E$. 
For $\xi\in L^2(\mathcal F_T,\Pp)$ set
\begin{equation}\label{eq:perturb-weak-derivative}
 k_\xi=\mathsf D_J\xi
 :=\nabla_J^{\mathbb P}\CE\xi,\qquad
 (\CE\xi)_t=\E[\xi\mid\mathcal F_t]
                    -\E[\xi\mid\mathcal F_0],\qquad
 \sigma_\xi=\norm{\xi-\E[\xi\mid\mathcal F_0]}_2.
\end{equation}
Thus $\mathsf D_J$ is the adjoint of terminal stochastic integration:
\begin{equation}\label{eq:perturb-duality}
 \E[\xi(I_Jv)_T]=\E\int k_\xi v\,\mathrm d\mu,
 \qquad \norm{k_\xi}_{\HJ}\leq\sigma_\xi,
 \qquad v\in \HJ.
\end{equation}
This is the stochastic derivative constructed in \eqref{eq:weak-definition}. The filtration may contain martingales
orthogonal to $\cIJ $. All unqualified space-time integrals
in this section are over $(0,T]\times E$.

\begin{theorem}[Response to a predictable compensator perturbation]
\label{thm:compensator-response}
Let $h$ be a real predictable field and suppose that, for deterministic
constants $B,K<\infty$,
\begin{equation}\label{eq:perturb-hypotheses}
 |h|\leq B\quad m\text{-a.e.},\qquad
 A_T:=\int h^2\,\mathrm d\mu\leq K\quad\Pp\text{-a.s.}
\end{equation}
Put $M=h*(J-\mu)$ and, for $z\in\mathbb C$, let
$L(z)=\mathcal E(zM)$ denote the complex stochastic exponential,
with initial value one. Then the following assertions hold.

\begin{enumerate}[label=\textnormal{(\roman*)}]
\item For every $z\in\mathbb C$, $L(z)$ is a square-integrable
martingale and
\begin{equation}\label{eq:perturb-exp-bound}
 \E|L_T(z)|^2\leq e^{|z|^2K},\qquad
 \norm{L(z)}_{\mathcal S^2}\leq2e^{|z|^2K/2},
 \qquad
 \norm{U}_{\mathcal S^2}^2:=\E\sup_{t\leq T}|U_t|^2.
\end{equation}
If $\varepsilon\in\R$ and $|\varepsilon|B<1$, then
$L^\varepsilon:=L(\varepsilon)$ is strictly positive and defines an
equivalent probability measure
\begin{equation}\label{eq:perturb-law}
 \frac{\mathrm d\Pp^\varepsilon}{\mathrm d\Pp}
       \bigg|_{\mathcal F_T}=L_T^\varepsilon.
\end{equation}
The measures agree on $\mathcal F_0$, and the compensator of $J$
under $\Pp^\varepsilon$ is
\begin{equation}\label{eq:perturb-compensator}
 \mu^\varepsilon=(1+\varepsilon h)\mu.
\end{equation}
\item For every such $\varepsilon$, the stochastic derivative satisfies
\begin{equation}\label{eq:perturb-density-derivative}
 \nabla_J^{\Pp}L^\varepsilon
           =\varepsilon hL^\varepsilon_-.
\end{equation}
For each real $\xi\in L^2(\mathcal F_T,\Pp)$,
\begin{equation}\label{eq:perturb-exact-response}
 \E^{\varepsilon}\xi-\E\xi
     =\varepsilon\E\int k_\xi hL^\varepsilon_-\,\mathrm d\mu.
\end{equation}
In particular,
\begin{equation}\label{eq:perturb-first-response}
 \left.\frac{\mathrm d}{\mathrm d\varepsilon}
           \E^\varepsilon\xi\right|_{\varepsilon=0}
       =\E\int k_\xi h\,\mathrm d\mu,
 \qquad
 \left|\E\int k_\xi h\,\mathrm d\mu\right|
       \leq\sqrt K\,\sigma_\xi.
\end{equation}
The first-order remainder has the explicit bound
\begin{equation}\label{eq:perturb-first-remainder}
 \left|\E^\varepsilon\xi-\E\xi
          -\varepsilon\E\int k_\xi h\,\mathrm d\mu\right|
 \leq 2|\varepsilon|\sqrt K\,\sigma_\xi
                   \sqrt{e^{\varepsilon^2K}-1}.
\end{equation}

\item Define the square-integrable martingales
\begin{equation}\label{eq:perturb-iterates}
 V_0\equiv1,\qquad
 V_n=(hV_{n-1,-})*(J-\mu),\quad n\geq1.
\end{equation}
For every $z\in\mathbb C$,
\begin{equation}\label{eq:perturb-entire-series}
 L(z)=\sum_{n=0}^{\infty}z^nV_n
 \quad\text{with absolute convergence in }\mathcal S^2,
\end{equation}
locally uniformly in $z$. For every $R>0$ and $n\geq0$,
\begin{equation}\label{eq:perturb-coefficient-bound}
 \norm{V_n(T)}_2\leq e^{R^2K/2}R^{-n},\qquad
 \norm{V_n}_{\mathcal S^2}\leq2e^{R^2K/2}R^{-n}.
\end{equation}
Consequently, $F_\xi(z):=\E[\xi L_T(z)]$ is an analytic function of
$z\in\mathbb C$. It equals $\E^\varepsilon\xi$ for real $z=\varepsilon$
with $|\varepsilon|B<1$, and has the power series expansion
\begin{equation}\label{eq:perturb-payoff-series}
 F_\xi(z)=\E\xi+\sum_{n=1}^{\infty}z^n c_n(\xi,h),\qquad
 c_n(\xi,h)=\E\int k_\xi hV_{n-1,-}\,\mathrm d\mu.
\end{equation}
In particular, $F_\xi^{(n)}(0)=n!c_n(\xi,h)$ for $n\geq1$.
If $K>0$, then
\begin{equation}\label{eq:perturb-optimized-coefficient}
 |c_n(\xi,h)|\leq\sigma_\xi
                   \left(\frac{eK}{n}\right)^{n/2},\qquad n\geq1.
\end{equation}

\item For $N\geq0$ and $R>|z|$,
\begin{equation}\label{eq:perturb-density-tail}
 \left\|L_T(z)-\sum_{n=0}^{N}z^nV_n(T)\right\|_2
 \leq e^{R^2K/2}
             \frac{(|z|/R)^{N+1}}{1-|z|/R}.
\end{equation}
The corresponding $\mathcal S^2$ bound has an additional factor two.
Moreover,
\begin{equation}\label{eq:perturb-payoff-tail}
 \left|F_\xi(z)-\E\xi-\sum_{n=1}^{N}z^n c_n(\xi,h)\right|
 \leq\sigma_\xi e^{R^2K/2}
             \frac{(|z|/R)^{N+1}}{1-|z|/R}.
\end{equation}
\end{enumerate}
\end{theorem}

\begin{proof}
Choose a bounded predictable version of $h$. Compensation shows that
the same bound holds at every atom of $J$, outside a null set. The
isometry and \eqref{eq:perturb-hypotheses} give
\[
 \E M_T^2=\E A_T\leq K,
 \qquad
 [M]_t=\int_{(0,t]\times E}h^2\,\mathrm dJ,
 \qquad [M]^c=0.
\]
The last two identities use both assumptions in
\eqref{eq:perturb-standing}. In particular,
$\sum_{s\leq T}|\Delta M_s|^2<\infty$ almost surely.
The stochastic exponential satisfies
\begin{equation}\label{eq:perturb-exp-equation}
 L_t(z)=1+\int_{(0,t]\times E}zh(s,x)L_{s-}(z)
                                  (J-\mu)(\mathrm ds,\mathrm dx).
\end{equation}
The process $A_t=\int_{(0,t]\times E}h^2\,\mathrm d\mu$ is
continuous. The product formula, first after localization, gives
\[
 \mathrm d|L_t(z)|^2
 =\int_E |L_{t-}(z)|^2
       \bigl(2\operatorname{Re}(z)h(t,x)+|z|^2h(t,x)^2\bigr)
                  (J-\mu)(\mathrm dt,\mathrm dx)
       +|z|^2|L_{t-}(z)|^2\,\mathrm dA_t.
\]
It follows that $e^{-|z|^2A_t}|L_t(z)|^2$ is a nonnegative local
martingale starting at one. Its supermartingale property yields,
for every stopping time $\tau\leq T$,
\[
 \E|L_\tau(z)|^2\leq e^{|z|^2K}.
\]
Applied to a localizing sequence, this bound makes $L(z)$ a true
square-integrable martingale. Doob's inequality proves
\eqref{eq:perturb-exp-bound}. If $|\varepsilon|B<1$, then every
jump multiplier $1+\varepsilon\Delta M$ is bounded away from zero.
The stochastic exponential is strictly positive on $[0,T]$; its
terminal value is a density of mean one and
$\E[L_T^\varepsilon\mid\mathcal F_0]=1$.

For a nonnegative predictable field $f$, optional sampling of the
density at the atoms of $J$, followed by compensation, gives
\begin{align*}
 \E^\varepsilon\int f\,\mathrm dJ
 &=\E\int f L_t^\varepsilon\,J(\mathrm dt,\mathrm dx)=\E\int f L_{t-}^\varepsilon(1+\varepsilon h)\, \mu(\mathrm dt,\mathrm dx)
 =\E^\varepsilon\int f(1+\varepsilon h)\,\mathrm d\mu.
\end{align*}
These equalities hold by monotone localization when the integrals are
infinite. In the last line we used the predictable projection
$L_-^\varepsilon$ of the density against the predictable random
measure. This proves \eqref{eq:perturb-compensator}.

\emph{The stochastic derivative and the first-order remainder.}
For every $U\in\mathcal S^2$,
\begin{equation}\label{eq:perturb-integrand-estimate}
 \E\int h^2|U_-|^2\,\mathrm d\mu
       \leq K\norm{U}_{\mathcal S^2}^2.
\end{equation}
Thus $hL_-^\varepsilon\in \HJ$, and
\eqref{eq:perturb-exp-equation} puts $L^\varepsilon-1$ in
$\cIJ $. The stochastic derivative gives
\eqref{eq:perturb-density-derivative}. Applying
\eqref{eq:perturb-duality} proves \eqref{eq:perturb-exact-response}.
All its terms are integrable by Cauchy--Schwarz and
\eqref{eq:perturb-integrand-estimate}.
Since $L^\varepsilon$ has mean one,
\[
 \E|L_T^\varepsilon-1|^2
      \leq e^{\varepsilon^2K}-1.
\]
Subtracting the first two terms in
\eqref{eq:perturb-exp-equation}, and using the isometry and Doob's
inequality, yields
\begin{align*}
 \norm{L_T^\varepsilon-1-\varepsilon M_T}_2^2
 =\varepsilon^2\E\int h^2|L_-^\varepsilon-1|^2\,\mathrm d\mu
\leq4\varepsilon^2K\E|L_T^\varepsilon-1|^2
 \leq4\varepsilon^2K(e^{\varepsilon^2K}-1).
\end{align*}
The variable $L^\varepsilon_T-1-\varepsilon M_T$ is the terminal value of
a square-integrable martingale started at zero, so its conditional
expectation given $\mathcal F_0$ vanishes. Hence
\[
 \E\bigl[\xi\,(L^\varepsilon_T-1-\varepsilon M_T)\bigr]
 =\E\bigl[(\xi-\E[\xi\mid\mathcal F_0])\,(L^\varepsilon_T-1-\varepsilon M_T)\bigr],
\]
and the Cauchy--Schwarz inequality gives 
\eqref{eq:perturb-first-remainder}. Its right-hand side is
$O(\varepsilon^2)$, which proves the derivative formula.
The  bound follows from \eqref{eq:perturb-duality} and
$\norm{h}_{\HJ}\leq\sqrt K$.

\paragraph{A series expansion.}
The operator
\[
 \mathcal T_h U=(hU_-)*(J-\mu)
\]
acts on the complexification of $\mathcal S^2$ and has norm at most
$2\sqrt K$, by \eqref{eq:perturb-integrand-estimate}, the isometry,
and Doob's inequality. Therefore the iterates in
\eqref{eq:perturb-iterates} belong to $\mathcal S^2$. For
$2|z|\sqrt K<1$, their series converges absolutely in that space,
solves $L=1+z\mathcal T_hL$, and hence equals $\mathcal E(zM)$.
If $K=0$, then $M=0$, $L=1$, and all assertions are immediate.
We may thus assume $K>0$.

For almost every sample path, the canonical product for the
stochastic exponential is
\begin{equation}\label{eq:perturb-canonical-product}
 L_t(z)=e^{zM_t}
       \prod_{0<s\leq t}(1+z\Delta M_s)e^{-z\Delta M_s}.
\end{equation}
This product is analytic in $z$, locally uniformly in $z$ and
uniformly in $t\leq T$. Indeed, the finitely many jumps for which
$|z\Delta M_s|$ can be large on a fixed disk contribute finitely
many analytic factors. For the remaining factors, the logarithmic
remainder is bounded by a constant times
$|z|^2|\Delta M_s|^2$, whose sum is finite. Also, the c\`adl\`ag
path $M$ is bounded on $[0,T]$. Zeros of the finite product cause
no difficulty.

On a sufficiently small disk, the preceding $\mathcal S^2$ series
converges almost surely absolutely and uniformly in $t$ and in
$z$ on every smaller  disk: this follows from
$\sum_n r^n\E\sup_{t\leq T}|V_n(t)|<\infty$ whenever
$2r\sqrt K<1$. Identify it with
\eqref{eq:perturb-canonical-product} first for a countable dense
set of parameters, using uniqueness of the stochastic exponential,
and then for all parameters on that disk by continuity. Thus its
coefficients $V_n(t)$ are the pathwise Taylor coefficients of
\eqref{eq:perturb-canonical-product}.

Cauchy's  formula on the circle $|z|=R$,
Minkowski's inequality and \eqref{eq:perturb-exp-bound} give
\[
 \norm{V_n(T)}_2
 \leq R^{-n}\frac1{2\pi}\int_0^{2\pi}
          \norm{L_T(Re^{i\theta})}_2\,\mathrm d\theta
 \leq e^{R^2K/2}R^{-n}.
\]
The same argument with the supremum over $t$ gives the
$\mathcal S^2$ bound in \eqref{eq:perturb-coefficient-bound}.
Choosing $R>|z|$ proves absolute convergence in
\eqref{eq:perturb-entire-series}; the limit is the pathwise
Taylor series of \eqref{eq:perturb-canonical-product}.
Summing the geometric majorant proves
\eqref{eq:perturb-density-tail} and its $\mathcal S^2$ counterpart.

For $n\geq1$, $V_n$ is a martingale starting at zero. Duality gives
\[
 \E[\xi V_n(T)]
   =\E\int k_\xi hV_{n-1,-}\,\mathrm d\mu
   =c_n(\xi,h).
\]
Taking expectations in the convergent series proves
\eqref{eq:perturb-payoff-series}. The density remainder has
conditional mean zero given $\mathcal F_0$, so
Cauchy--Schwarz proves \eqref{eq:perturb-payoff-tail} with
$\sigma_\xi$. Finally, for $n\geq1$, minimizing
$e^{R^2K/2}R^{-n}$ over $R>0$ gives $R=\sqrt{n/K}$ and proves
\eqref{eq:perturb-optimized-coefficient}.
\end{proof}

\begin{remark}\label{rem:perturb-scope}
The direction $h$ is a fixed predictable field on the underlying
space. Under $\Pp^\varepsilon$ the same field and the same random
measure $\mu$ occur in \eqref{eq:perturb-compensator}; their
distributions may change. Thus the theorem allows dependence on
the entire observed past, without asserting existence or uniqueness
for a separately specified family of point-process laws. It requires
neither finite activity nor a density of $\mu$ with respect to
Lebesgue measure. The bound $A_T\leq K$ is an almost-sure
deterministic bound, stronger than $\E A_T<\infty$. The theorem
concerns multiplicative, absolutely continuous perturbations.
Complex parameters provide analytic continuation only, not
probability measures. On their common domain, the first-order
coefficient can also be expressed through a Malliavin derivative
by the predictable-projection theorem of the preceding section.
\end{remark}
\begin{example}[A feedback perturbation]\label{ex:feedback}
Let $N$ be a pure birth process in its natural filtration, with
$N_0=n_0$ and rates $0<\ell(n)\leq\Lambda$. Thus
$\mu(\mathrm dt,\mathrm dx)=\ell(N_{t-})\mathrm dt\,\delta_1(\mathrm dx)$.
Fix $\beta>0$, a measurable function $b:\R_+\to[-1,1]$, and set
\[
 R_t=\int_{(0,t]}e^{-\beta(t-s)}\,\mathrm dN_s,
 \qquad h(t,1)=b(R_{t-}).
\]
Then $A_T\leq\Lambda T$. For $|\varepsilon|<1$, the perturbed
law has intensity
$(1+\varepsilon b(R_{t-}))\ell(N_{t-})$.
For any bounded $f:\N\to\R$, define
$u(t,n)=\E[f(N_T)\mid N_t=n]$ under the unperturbed law.
The bounded function $u$ is continuous in $t$, and $Y_t=u(t,N_t)$ is a
bounded martingale. Proposition~\ref{prop:identification}, applied to
the functional $F(t,j)=u\bigl(t,n_0+j((0,t]\times E)\bigr)$, gives
$k_{f(N_T)}(t,1)=u(t,N_{t-}+1)-u(t,N_{t-})$.
Hence
\begin{equation}\label{eq:perturb-feedback-example}
 \left.\frac{\mathrm d}{\mathrm d\varepsilon}
                  \E^\varepsilon f(N_T)\right|_{\varepsilon=0}
 =\E\int_0^T
   \bigl[u(t,N_{t-}+1)-u(t,N_{t-})\bigr]
                b(R_{t-})\ell(N_{t-})\,\mathrm dt.
\end{equation}
The error estimate \eqref{eq:perturb-first-remainder} holds with
$K=\Lambda T$, and all higher coefficients follow from
\eqref{eq:perturb-payoff-series}. The choice of a merely bounded
$f$ includes indicators, and the perturbed intensity can
depend on more of the history than the current value of $N$.
\end{example}
The next proposition identifies the stochastic derivative under an
equivalent change of measure. Its assumptions concern two specific martingales in the integral range; the predictable representation property is not assumed.
\begin{proposition}[Change of law for the stochastic derivative]
\label{prop:perturb-change-law}
Let $Z$ be a square-integrable $\Pp$-martingale with $Z_0=1$ and
$Z_T>0$ almost surely, and define
$\mathrm d\mathbb Q=Z_T\,\mathrm d\Pp$ on $\mathcal F_T$.
Let $Y$ be a bounded $\mathbb Q$-martingale. Assume that
\begin{equation}\label{eq:perturb-range-hypotheses}
 Z-1\in\cIJ ,\qquad ZY-Y_0\in\cIJ ,
\end{equation}
and put
\[
 q=\nabla_J^{\Pp}(Z-1),\qquad
 r=\nabla_J^{\Pp}(ZY-Y_0).
\]
Then $Z_-+q>0$ for $m$-almost every $(\omega,t,x)$,
and the $\mathbb Q$-compensator of $J$ is
\begin{equation}\label{eq:perturb-transport-compensator}
 \mu^{\mathbb Q}=\frac{Z_-+q}{Z_-}\,\mu.
\end{equation}
Moreover, $Y-Y_0$ belongs to the square-integrable integral range
under $\mathbb Q$, and
\begin{equation}\label{eq:perturb-transport-derivative}
 \boxed{\displaystyle
 \nabla_J^{\mathbb Q}Y
       =\frac{r-Y_-q}{Z_-+q}.}
\end{equation}
The quotient is an equality of predictable energy classes. More
precisely, with $g=(r-Y_-q)/(Z_-+q)$,
\begin{equation}\label{eq:perturb-transport-energy}
 \E_{\mathbb Q}\int g^2\,\mathrm d\mu^{\mathbb Q}
 =\E\int\frac{(r-Y_-q)^2}{Z_-+q}\,\mathrm d\mu
 =\E_{\mathbb Q}|Y_T-Y_0|^2.
\end{equation}
If, in addition, $0<c\leq Z_t\leq C<\infty$ for all $t\leq T$,
then
\begin{equation}\label{eq:perturb-transport-comparison}
 \frac1C\norm{r-Y_-q}_{\HJ}^2
 \leq
 \|g\|^2_{\mathcal L^2_{\mathbb Q}(\mu^{\mathbb Q})}
 \leq\frac1c\norm{r-Y_-q}_{\HJ}^2.
\end{equation}
If $J-\mu$ has the predictable representation property under
$\Pp$, then $J-\mu^{\mathbb Q}$ has it under $\mathbb Q$.
\end{proposition}

\begin{proof}
Bayes' formula shows that $U:=ZY$ is a $\Pp$-martingale.
It is square-integrable because $Y$ is bounded and $Z_T\in L^2$.
The range assumptions and the definition of the stochastic derivative give
\begin{equation}\label{eq:perturb-two-representations}
 Z=1+q*(J-\mu),\qquad
 U=Y_0+r*(J-\mu).
\end{equation}
The strictly positive terminal density implies that $Z$ and $Z_-$
are strictly positive on $[0,T]$, up to indistinguishability.
For example, under $\mathbb Q$ the reciprocal density is the
c\`adl\`ag martingale
$\E_{\mathbb Q}[Z_T^{-1}\mid\mathcal F_t]=Z_t^{-1}$;
localizing it also bounds $Z$ away from zero.
At an atom $(t,x)$ of $J$, atomlessness of $\mu$ gives
\[
 Z_t=Z_{t-}+q(t,x)>0.
\]
Compensating the indicator of the predictable set
$\{Z_-+q\leq0\}$ proves positivity $m$-almost everywhere.
The density calculation used in the preceding proof, now with
$Z_t=Z_-+q$ at the atoms, gives
\[
 \E_{\mathbb Q}\int f\,\mathrm dJ
 =\E\int f(Z_-+q)\,\mathrm d\mu
 =\E_{\mathbb Q}\int f\frac{Z_-+q}{Z_-}\,\mathrm d\mu
\]
for every nonnegative predictable $f$. This proves
\eqref{eq:perturb-transport-compensator}. In particular,
$\mu^{\mathbb Q}$ is time-atomless.
On the predictable $\sigma$-field, the control measure of
$\mu^{\mathbb Q}$ under $\mathbb Q$,
$m^{\mathbb Q}(B)=\E^{\mathbb Q}\int\mathbf 1_B\,\mathrm d\mu^{\mathbb Q}$,
has the finite strictly positive density $Z_-+q$ with respect to $m$,
and is therefore $\sigma$-finite.
At the same atom the quotient $Y=U/Z$ satisfies
\begin{equation}\label{eq:perturb-jump-quotient}
 \Delta Y_t
 =\frac{U_{t-}+r(t,x)}{Z_{t-}+q(t,x)}-Y_{t-}
 =\frac{r(t,x)-Y_{t-}q(t,x)}{Z_{t-}+q(t,x)}
 =g(t,x).
\end{equation}
There are no jumps of $Y$ outside the atoms of $J$.
The two martingales in \eqref{eq:perturb-two-representations}
have zero continuous quadratic variation. Applying the
semimartingale change-of-variable formula to $U/Z$, locally on
sets where $Z$ is bounded away from zero, shows that $[Y]^c=0$.
This pathwise property is preserved under the equivalent measure
$\mathbb Q$. Since $Y$ is a bounded $\mathbb Q$-martingale,
compensation and \eqref{eq:perturb-jump-quotient} give
\begin{align*}
 \E_{\mathbb Q}\int g^2\,\mathrm d\mu^{\mathbb Q}
 &=\E_{\mathbb Q}\sum_{0<t\leq T}(\Delta Y_t)^2 =\E_{\mathbb Q}[Y]_T
 =\E_{\mathbb Q}|Y_T-Y_0|^2<\infty.
\end{align*}
Consequently $g*(J-\mu^{\mathbb Q})$ is a square-integrable
$\mathbb Q$-martingale. Subtracting it from $Y-Y_0$ leaves a
continuous martingale with zero quadratic variation, hence the
zero martingale. The stochastic derivative under $\mathbb Q$ now gives
\eqref{eq:perturb-transport-derivative}.

For a nonnegative predictable integrand against $\mu$, the
predictable projection of $Z_T$ is $Z_-$. Thus
\[
 \E_{\mathbb Q}\int g^2\,\mathrm d\mu^{\mathbb Q}
 =\E\int Z_-g^2\frac{Z_-+q}{Z_-}\,\mathrm d\mu
 =\E\int\frac{(r-Y_-q)^2}{Z_-+q}\,\mathrm d\mu,
\]
which proves \eqref{eq:perturb-transport-energy}.
Under the additional bounds on $Z$, compensation at its jump
times gives $c\leq Z_-+q\leq C$ $m$-almost everywhere.
Since $r-Y_-q\in \HJ$, this implies
\eqref{eq:perturb-transport-comparison}.

Finally, if $J-\mu$ itself has the predictable representation
property under $\Pp$, the range
conditions hold automatically for every bounded
$\mathbb Q$-martingale $Y$. Such martingales therefore have a
representation under $\mathbb Q$. Truncation of terminal values
and closedness of the $\mathbb Q$-integral range extend the result
to all square-integrable $\mathbb Q$-martingales, after subtraction
of their initial values.
\end{proof}

For the density of Theorem~\ref{thm:compensator-response},
$q=\varepsilon hL_-^\varepsilon$, so the denominator in
\eqref{eq:perturb-transport-derivative} becomes
$L_-^\varepsilon(1+\varepsilon h)$. The range condition on
$L^\varepsilon Y-Y_0$ is still needed unless $J-\mu$ itself has
the predictable representation property under $\Pp$. 
The formula describes how the stochastic
derivative changes with the law; the two derivatives are generally
different operators on different predictable energy spaces.

\begin{remark}[Relation with existing perturbation results]
\label{rem:perturb-literature}
The change of compensator is the classical likelihood construction
for random measures; see \cite{JS03}. Bismut's jump-process
calculus already uses state-dependent stochastic-exponential
changes of intensity with square-energy estimates, and relates
mark perturbations to martingale representation
\cite[Sections~1--3]{Bismut83}; it also allows infinite activity.
The first-order response follows by differentiating
the likelihood at zero and applying the
stochastic-integral duality~\eqref{eq:perturb-duality}.
Decreusefond
\cite{Decreusefond98} studied likelihood sensitivity
for square-integrable functionals of marked point processes with
random predictable compensators, using an $L^2$ iterated-integral
expansion of the density, assuming finite
counts with a second moment and a deterministic bound on the
time density of the perturbation's quadratic energy. Theorem~\ref{thm:compensator-response}
instead assumes \eqref{eq:perturb-hypotheses}, allowing infinite
activity and a singular continuous energy clock. For deterministic
Poisson intensities, signed perturbations, $L^2$ response formulas,
and convergent expansions are treated in
\cite{Last2014Perturbation}. Iterated-integral
expansions of stochastic exponentials hold for general
semimartingales \cite[Theorem~2.8]{Jamshidian11}, and preservation
of martingale representation under equivalent laws is classical
\cite{Duffie85}. The results above give  an analytic $\mathcal S^2$ expansion
under assumption \eqref{eq:perturb-hypotheses}, the coefficient and remainder
bounds \eqref{eq:perturb-coefficient-bound}--\eqref{eq:perturb-payoff-tail},
and the law-dependent operator identity
\eqref{eq:perturb-transport-derivative} under the two range
conditions \eqref{eq:perturb-range-hypotheses}. Their proofs
allow random time-atomless compensators and predictable feedback,
and use only the stochastic derivative of the terminal variable.

Intensity perturbations 
were also used by Carlen and Pardoux \cite{CarlenPardoux1990}
to derive integration-by-parts
formulas on Poisson space; see Remark \ref{rem:jump-time}.
\end{remark}

\section{Entropy and the stochastic derivative}
\label{sec:entropy}

We give a quantitative application of the stochastic derivative to
relative entropy. The derivative identifies the part of a likelihood
martingale represented by the compensated random measure. The
remaining entropy admits an exact formula and two-sided bounds in
terms of the orthogonal martingale residual. No representation
property is assumed.

Fix $T<\infty$ and retain the filtered space and assumptions of
Section~\ref{sec:compensator-perturbations}, including Assumption \ref{ass:main}.
In particular, $J(\{t\}\times E)\leq1$ and
$\mu(\{t\}\times E)=0$ simultaneously for all $t$, outside a
null set. Recall that
$\cIJ =\{h*(J-\mu):h\in \HJ\}$, with the
terminal $L^2(\Pp)$ norm. For $\xi\in L^2(\mathcal F_T,\Pp)$,
$\nabla_J^{\mathbb P}\CE\xi$ denotes the
stochastic derivative in \eqref{eq:perturb-weak-derivative}.
All unqualified integrals below are over $(0,T]\times E$, and
expectations are under $\Pp$ unless specified otherwise. Set
\begin{equation}\label{eq:entropy-functions}
 \ell(r)=r\log r-r+1,\qquad
 B(u,v)=u\log(u/v)-u+v,\qquad r,u,v>0.
\end{equation}

\begin{theorem}[Quantitative entropy decomposition]
\label{thm:entropy-decomposition}
Let $\zeta$ be $\mathcal F_T$-measurable and suppose that, for
deterministic constants $0<a\leq b<\infty$,
\begin{equation}\label{eq:entropy-density-assumptions}
 a\leq\zeta\leq b\quad\Pp\text{-a.s.},\qquad
 \E[\zeta\mid\mathcal F_0]=1.
\end{equation}
 By the martingale representation formula (Theorem~\ref{thm:representation}), $Z_t=\E[\zeta\mid\mathcal F_t]$ has a unique decomposition
\begin{equation}\label{eq:entropy-orthogonal-decomposition}
Z=1+K+N\qquad{\rm with}
\qquad K=(\nabla^{\Pp}_J Z)*(J-\mu),\qquad
N\perp \cIJ.
\end{equation}
Define $\mathrm d\mathbb Q=\zeta\,\mathrm d\Pp$ on $\mathcal F_T$. Then the
following assertions hold.
\begin{enumerate}[label=\textnormal{(\roman*)}]
\item The $\mathbb Q$-compensator of $J$ is
$\mu^{\mathbb Q}=r\mu$ with $r=1+\mathsf \nabla^{\Pp}_J Z/Z_-$ and
\begin{equation}\label{eq:entropy-ratio-bounds}
 a\leq Z_-+\nabla^{\Pp}_J Z\leq b,\qquad \frac{a}{b}\leq r= 1+\frac{\nabla^{\Pp}_J Z}{Z_-}\leq \frac{b}{a}
 \qquad  m\text{-a.e.}
\end{equation}
\item Entropy decomposition:
\begin{equation}\label{eq:entropy-decomposition}
 H(\mathbb Q\mid\Pp)=\E[\zeta\log\zeta]
   =\E\int Z_-\ell(r)\,\mathrm d\mu+\mathcal R_J(\zeta)
   =\E_{\mathbb Q}\int\ell(r)\,\mathrm d\mu+\mathcal R_J(\zeta),
\end{equation}
where the   remainder is nonnegative:
\begin{equation}\label{eq:entropy-exact-remainder}
 \mathcal R_J(\zeta)
   =\frac12\E\int_0^T\frac{\mathrm d[N]^c_t}{Z_{t-}}
      +\E\sum_{0<t\leq T}
             B\bigl(Z_t,Z_{t-}+\Delta K_t\bigr)\geq 0.
\end{equation}
Here $[N]^c$ is the continuous part of the quadratic variation;
the sum includes every jump of $Z$ or $K$.

\item All terms in \eqref{eq:entropy-decomposition} and
\eqref{eq:entropy-exact-remainder} are finite, and
\begin{equation}\label{eq:entropy-remainder-bounds}
 \frac{1}{2b}\E[N_T^2]
       \leq\mathcal R_J(\zeta)
       \leq\frac{1}{2a}\E[N_T^2].
\end{equation}
The residual energy is determined by the stochastic derivative $\nabla^{\Pp}_J Z$:
\begin{equation}\label{eq:entropy-residual-energy}
 \E[N_T^2]=\E[(\zeta-1)^2]-\norm{\nabla^{\Pp}_J Z}_{\HJ}^2.
\end{equation}
\end{enumerate}
\end{theorem}

\begin{proof}
The density martingale has a c\`adl\`ag version for which
$a\leq Z_t,Z_{t-}\leq b$ simultaneously on $[0,T]$, outside one
null set, and $Z_0=1$. The orthogonal decomposition and
\eqref{eq:entropy-residual-energy} follow from
Theorem~\ref{thm:representation}, applied to the jump integral
alone, or directly from \eqref{eq:perturb-duality} and its
isometry. Thus, for every $h\in \HJ$,
\begin{equation}\label{eq:entropy-orthogonality}
 \E[N_T(h*(J-\mu))_T]=0.
\end{equation}
Let $f$ be a bounded nonnegative predictable field supported on
a predictable set of finite $m$-measure. Compensation and
the covariance identity for square-integrable martingales give
\begin{align}
 \E\int fZ_t\,\mathrm dJ
 &=\E\int fZ_-\,\mathrm d\mu+\E\int f\Delta Z_t\,\mathrm dJ
 \notag\\
 &=\E\int fZ_-\,\mathrm d\mu
       +\E[(Z_T-1)(f*(J-\mu))_T]
 =\E\int f(Z_-+q)\,\mathrm d\mu,
 \label{eq:entropy-jump-projection}
\end{align}
where $q=\nabla^{\Pp}_J(\CE\zeta)=\nabla^{\Pp}_J Z.$
The covariance has only its jump part, since the compensated
integral is purely discontinuous; its jump at an atom $(t,x)$
is $f(t,x)$. Integrability follows from Cauchy--Schwarz for
quadratic variations. Since $a\leq Z_t\leq b$,
\eqref{eq:entropy-jump-projection} lies between
$a\E\int f\,\mathrm d\mu$ and $b\E\int f\,\mathrm d\mu$.
Testing on predictable sets and using $\sigma$-finiteness gives
$a\leq Z_-+q\leq b$ $m$-almost everywhere. Division by
$Z_-\in[a,b]$ proves \eqref{eq:entropy-ratio-bounds}.
Choose a finite predictable version of $q$ and set it equal to
zero on the exceptional predictable set. Compensation shows
that this change affects no atom of $J$, outside a null set.
In particular the same bounds hold at all such atoms.

For every nonnegative predictable field $f$, optional projection
at the atoms and predictable projection against $\mu$ yield
\[
 \E_{\mathbb Q}\int f\,\mathrm dJ
  =\E\int fZ_t\,\mathrm dJ
  =\E\int f(Z_-+q)\,\mathrm d\mu
  =\E_{\mathbb Q}\int fr\,\mathrm d\mu.
\]
These identities follow first on the finite-control-measure
sets above and then by monotone convergence. In the last step
we use the predictable projection of the terminal variable
$\zeta$, which is $Z_-$; $\mu$ is predictable and time-atomless.
This proves $\mu^{\mathbb Q}=r\mu$.
For completeness, the predictable-projection identity can also
be obtained directly. For $A=f*\mu$ with $f$ as above, $A$ is
continuous, adapted, increasing and integrable. Stop $A$ at
bounded levels and apply integration by parts to $ZA$:
the stochastic integral has mean zero, giving
$\E[\zeta A_T]=\E\int_0^T Z_-\,\mathrm dA$ for the stopped
process. Monotone convergence removes the stopping.

\emph{The entropy identity.}
Put $\phi(u)=u\log u$, so that
$B(u,v)=\phi(u)-\phi(v)-\phi'(v)(u-v)$ and
$\phi''(u)=1/u$. The compensated integral $K$ is purely
discontinuous, hence $[Z]^c=[N]^c$. It\^o's formula and $Z_0=1$
give
\begin{equation}\label{eq:entropy-ito}
 H(\mathbb Q\mid\Pp)
   =\frac12\E\int_0^T\frac{\mathrm d[N]^c_t}{Z_{t-}}
       +\E\sum_{0<t\leq T}B(Z_t,Z_{t-}).
\end{equation}
Indeed $\phi'(Z_-)$ is bounded, so the stochastic integral in
It\^o's formula is a square-integrable martingale of mean zero.
Taylor's formula on $[a,b]$ gives
\begin{equation}\label{eq:entropy-curvature}
 \frac{(u-v)^2}{2b}\leq B(u,v)
                 \leq\frac{(u-v)^2}{2a},
 \qquad u,v\in[a,b],
\end{equation}
which also justifies integrability of the other terms.

Define $\widehat Z_t=Z_{t-}+\Delta K_t$. At an atom $(t,x)$
of $J$, time simplicity and atomlessness of $\mu$ give
$\widehat Z_t=Z_{t-}+q(t,x)$; at other times
$\widehat Z_t=Z_{t-}$. Thus $\widehat Z_t\in[a,b]$
simultaneously outside a null set, and
$Z_t-\widehat Z_t=\Delta N_t$. The identity
\begin{align}
 B(Z_t,Z_{t-})
 &=B(Z_t,\widehat Z_t)+B(\widehat Z_t,Z_{t-})
 +
       \bigl(\log\widehat Z_t-\log Z_{t-}\bigr)\Delta N_t
 \label{eq:entropy-three-point}
\end{align}
separates the jump terms in \eqref{eq:entropy-ito}.

The last multiplier in \eqref{eq:entropy-three-point} vanishes
away from the atoms of $J$ and equals $\log r(t,x)$ at an atom.
The mean value theorem and \eqref{eq:entropy-ratio-bounds}
give
\[
 |\log r|=|\log(Z_-+q)-\log Z_-|\leq |q|/a,
\]
so $\log r\in \HJ$. Moreover,
\[
 \E\sum_{\substack{0<t\leq T\\J(\{t\}\times E)=1}}
       |\log r(t,x_t)\Delta N_t|
 \leq
 \left(\E\int(\log r)^2\,\mathrm d\mu\right)^{1/2}
       \left(\E[N]_T\right)^{1/2}<\infty,
\]
where $x_t$ denotes the unique mark at that time. Therefore
the covariance identity and \eqref{eq:entropy-orthogonality}
imply
\begin{equation}\label{eq:entropy-cross-term}
 \E\sum_{0<t\leq T}
       \bigl(\log\widehat Z_t-\log Z_{t-}\bigr)\Delta N_t
   =\E[N_T(\log r*(J-\mu))_T]=0.
\end{equation}
For the middle term of \eqref{eq:entropy-three-point},
compensation gives
\begin{align}
 \E\sum_{0<t\leq T}B(\widehat Z_t,Z_{t-})
 &=\E\int B(Z_-+q,Z_-)\,\mathrm d\mu
 =\E\int Z_-\ell(r)\,\mathrm d\mu
   =\E_{\mathbb Q}\int\ell(r)\,\mathrm d\mu.
 \label{eq:entropy-compensator-cost}
\end{align}
The last equality again uses predictable projection. The
first expectation is at most
$\norm{q}_{\HJ}^2/(2a)$ by
\eqref{eq:entropy-curvature}. Also
\[
 \E\sum_{0<t\leq T}B(Z_t,\widehat Z_t)
       \leq\frac{1}{2a}\E\sum_{0<t\leq T}(\Delta N_t)^2<\infty.
\]
Consequently all sums in \eqref{eq:entropy-three-point} can
be integrated separately. Substitution of
\eqref{eq:entropy-cross-term} and
\eqref{eq:entropy-compensator-cost} into
\eqref{eq:entropy-ito} proves
\eqref{eq:entropy-decomposition}--\eqref{eq:entropy-exact-remainder}.

Apply \eqref{eq:entropy-curvature} to
$u=Z_t$, $v=\widehat Z_t$, and use
$u-v=\Delta N_t$ and $1/b\leq1/Z_-\leq1/a$. The remainder is
bounded below and above by, respectively, $1/(2b)$ and
$1/(2a)$ times
\[
 \E\left([N]^c_T+\sum_{0<t\leq T}(\Delta N_t)^2\right)
       =\E[N]_T=\E[N_T^2].
\]
This proves \eqref{eq:entropy-remainder-bounds} and completes
the proof.
\end{proof}

\begin{corollary}[Entropy equality and predictable representation]
\label{cor:entropy-representation}
For a density satisfying \eqref{eq:entropy-density-assumptions},
\begin{equation}\label{eq:entropy-equality}
 H(\mathbb Q\mid\Pp)=\E_{\mathbb Q}\int\ell(r)\,\mathrm d\mu
\end{equation}
holds if and only if $Z-1\in\cIJ $.
Furthermore, the following statements are equivalent:
\begin{enumerate}[label=\textnormal{(\roman*)}]
\item $J-\mu$ has the square-integrable predictable
representation property under $\Pp$: every $M\in\cM$ is
$h*(J-\mu)$ for some $h\in \HJ$.
\item For every
$\zeta$ satisfying \eqref{eq:entropy-density-assumptions}, Equation~\eqref{eq:entropy-equality} holds with constants $a,b$ allowed to depend on $\zeta$.
\end{enumerate}
This is the representation property of $J-\mu$ itself.
\end{corollary}
On the canonical configuration space, with $\mathcal F_T$
generated by $J$, $H(\mathbb Q\mid\Pp)$ is the relative
entropy of the corresponding laws of the integer-valued
random measure.
\begin{proof}
The first assertion follows from
\eqref{eq:entropy-remainder-bounds}: equality is equivalent
to $\E[N_T^2]=0$, hence to $N=0$ up to indistinguishability.
It also proves \textnormal{(i)}$\Rightarrow$\textnormal{(ii)}.
\noindent
Conversely assume \textnormal{(ii)}. For a bounded
$\mathcal F_T$-measurable variable $F$, put
$G=F-\E[F\mid\mathcal F_0]$. Choose a nonzero $\varepsilon$
with $|\varepsilon|\norm{G}_{\infty}<1$ and set
$\zeta=1+\varepsilon G$. This density satisfies
\eqref{eq:entropy-density-assumptions}, and
$Z-1=\varepsilon\CE F$. The first assertion gives
$\CE F\in\cIJ $.
For a general $\xi\in L^2(\mathcal F_T,\Pp)$, bounded
truncations $F_n$ converge to $\xi$ in $L^2$.
Conditional centering is an orthogonal projection, so
$\CE F_n\to\CE\xi$ in the terminal martingale norm.
The integral range is closed by the isometry, and therefore
$\CE\xi\in\cIJ $. Since for every $M\in\cM,$ 
$M=\CE M_T$, this yields \textnormal{(i)}.
\end{proof}

\paragraph{Relation with  classical likelihood formulas.}
The likelihood construction and change of compensator are studied in Br\'emaud and Jacod
\cite[Section~3.5, (32)-(33)]{BremaudJacod77}.
In their setting of nonexplosive
marked point processes under the natural filtration
enlarged by initial information, the remainder
above vanishes. 
The entropy--compensator inequality is given in
\cite[Theorem~2.11]{Leonard12}. The result above is a
quantitative application of the stochastic derivative: for
densities bounded above and bounded away from zero it identifies the exact
orthogonal remainder and proves
\eqref{eq:entropy-remainder-bounds}, without assuming
representation. No deterministic bound on
$\int(r-1)^2\,\mathrm d\mu$ is required.

\section{A pathwise refinement of Mecke's formula}
\label{sec:pathwise-mecke}
Mecke's formula \cite{Mecke67} is a fundamental identity  for Poisson point processes which relates the expected sum of a function over all points of the process to an integral over the state space where a point is added to the process \cite[Sec. 1]{Last14}.
 In this section, we will derive a new pathwise identity for point processes and show that it implies the Mecke formula.

Let $S=(0,T]\times E$, where $E$ is a standard Borel space, and
let $\lambda$ be a deterministic $\sigma$-finite measure on $S$
such that
\begin{equation}\label{eq:mecke-time-atomless}
 \lambda(\{t\}\times E)=0,\qquad 0<t\leq T.
\end{equation}
For a Poisson random measure $\eta$ with intensity $\lambda$, the
Mecke formula \cite{Mecke67} states that
\begin{equation}\label{eq:mecke-classical}
 \E\int_S H(\eta,x)\,\eta(\mathrm dx)
 =
 \E\int_S H(\eta+\delta_x,x)\,\lambda(\mathrm dx)
\end{equation}
for every nonnegative measurable $H$.
Here $x=(t,z)$ denotes a time--mark pair. In this section the
compensator is deterministic; a density with respect to Lebesgue
time is not required.

Fix measurable sets $S_n\uparrow S$ with $\lambda(S_n)<\infty$.
Let $\mathcal N_\lambda$ be the space of counting measures $j$
on $S$ such that $j(S_n)<\infty$ for every $n$, equipped with the
$\sigma$-field generated by all evaluation maps $j\mapsto j(C)$.
Test functions $H$ are defined on $\mathcal N_\lambda\times S$,
so that $H(j+\delta_x,x)$ is defined even when insertion produces
multiple points. Insertion is measurable, since evaluation on a
measurable set $D$ gives
$(j+\delta_x)(D)=j(D)+\one_D(x)$.
The configurations in the pathwise identity below
are time-simple. A random measure with compensator $\lambda$
belongs to $\mathcal N_\lambda$ almost surely, since
$\E J(S_n)=\lambda(S_n)<\infty$.
For an integrable function $v$, denote
\[
 j_t(v)=\int_{(0,t]\times E}v\,\mathrm dj,\qquad
 \lambda_t(v)=\int_{(0,t]\times E}v\,\mathrm d\lambda,\qquad
 j(v)=j_T(v),\qquad \lambda(v)=\lambda_T(v).
\]

\begin{proposition}[Pathwise identity and stochastic derivative]
\label{prop:mecke-pathwise}
Let $f:S\to[0,\infty)$ and $g:S\to\R$ be bounded measurable
functions vanishing outside a measurable set $C$ with
$\lambda(C)<\infty$, and put $q=e^{-f}$. For a time-simple
configuration $j\in\mathcal N_\lambda$ with $j(C)<\infty$, define
\begin{align}
 M_t^f(j)
 &=
 \exp\left\{-j_t(f)+
       \int_{(t,T]\times E}(q-1)\,\mathrm d\lambda\right\},
                                                   \label{eq:mecke-M}\\
 B_t^{f,g}(j)
 &=j_t(g)-\lambda_t(gq),\qquad
 V_t^{f,g}(j)=M_t^f(j)B_t^{f,g}(j),                  \label{eq:mecke-B}\\
 K_{f,g}(t,z,j)
 &=M_{t-}^f(j)
   \left[
      B_{t-}^{f,g}(j)\bigl(q(t,z)-1\bigr)
          +g(t,z)q(t,z)
   \right].                                       \label{eq:mecke-K}
\end{align}
Then, simultaneously for every $t\in[0,T]$,
\begin{equation}\label{eq:mecke-pathwise-process}
 V_t^{f,g}(j)
 =
 \int_{(0,t]\times E}
       K_{f,g}(s,z,j)\,(j-\lambda)(\mathrm ds,\mathrm dz).
\end{equation}
The right-hand side is the difference of two absolutely convergent
pathwise integrals. In particular,
\begin{equation}\label{eq:mecke-pathwise-terminal}
 e^{-j(f)}\bigl(j(g)-\lambda(ge^{-f})\bigr)
 =
 \int_S K_{f,g}(t,z,j)\,(j-\lambda)(\mathrm dt,\mathrm dz).
\end{equation}
Moreover, the jump derivative satisfies
\begin{equation}\label{eq:mecke-functional-derivative}
 \nabla_J V^{f,g}(t,z,j)=K_{f,g}(t,z,j).
\end{equation}

Let now $J$ be an adapted integer-valued random measure on a
filtered probability space satisfying the usual conditions, with
predictable compensator $\lambda$, and suppose that
$J(\{t\}\times E)\leq1$ simultaneously for all $t$, almost surely.
Then $K_{f,g}(\cdot,\cdot,J)$ is predictable and belongs to
${\mathcal L^2_{\Pp}(\lambda)} $.
The process $V^{f,g}(J)$ is a square-integrable martingale
starting from zero, and
\begin{equation}\label{eq:mecke-closed-derivative}
 \nabla_J^{\Pp} V^{f,g}(J)=K_{f,g}(\cdot,\cdot,J)
 \quad\text{in }{\mathcal L^2_{\Pp}(\lambda)} .
\end{equation}
\end{proposition}

\begin{proof}
Since $f$ and $g$ vanish outside $C$, the paths $j_t(f)$ and
$j_t(g)$ have only finitely many jumps and finite variation.
Condition~\eqref{eq:mecke-time-atomless} implies that
$\lambda_t(q-1)$ and $\lambda_t(gq)$ are continuous functions
of finite variation. Thus $M^f$, $B^{f,g}$, and $V^{f,g}$ are
c\`adl\`ag paths of finite variation. They depend only on the
configuration up to time $t$; the integral over $(t,T]$ in
\eqref{eq:mecke-M} is deterministic.
At a point $(t,z)$ of $j$, 
\begin{equation}\label{eq:mecke-jump-increments}
 \Delta M_t^f=M_{t-}^f\bigl(q(t,z)-1\bigr),
 \qquad
 \Delta B_t^{f,g}=g(t,z).
\end{equation}
The continuous finite-variation part of $M^f$ is
$-M_{t-}^f\,\mathrm d\lambda_t(q-1)$.
Consequently, as identities of signed Stieltjes measures in time,
\begin{align}
 \mathrm dM_t^f
 &=
 \int_E M_{t-}^f\bigl(q(t,z)-1\bigr)
                   (j-\lambda)(\mathrm dt,\mathrm dz),
                                                   \label{eq:mecke-dM}\\
 \mathrm dB_t^{f,g}
 &=
 \int_E g(t,z)\,j(\mathrm dt,\mathrm dz)
 -
 \int_E g(t,z)q(t,z)\,\lambda(\mathrm dt,\mathrm dz).
\label{eq:mecke-dB}
\end{align}
The use of the left limit in the continuous terms is legitimate
because $\lambda$ charges none of the finitely many relevant
jump times.
Applying the product formula for c\`adl\`ag finite-variation paths, we obtain
\[
 \mathrm d(M^fB^{f,g})_t
 =
 B_{t-}^{f,g}\,\mathrm dM_t^f
 +M_{t-}^f\,\mathrm dB_t^{f,g}
 +\mathrm d[M^f,B^{f,g}]_t,
\]
where the covariation consists only of the jump products. By
\eqref{eq:mecke-jump-increments},
\[
 \mathrm d[M^f,B^{f,g}]_t
 =
 \int_E M_{t-}^f\bigl(q(t,z)-1\bigr)g(t,z)
                                      j(\mathrm dt,\mathrm dz).
\]
Substituting \eqref{eq:mecke-dM}--\eqref{eq:mecke-dB}, the
coefficient of $j(\mathrm dt,\mathrm dz)$ becomes
\[
 M_{t-}^f
 \left[B_{t-}^{f,g}(q-1)+g+(q-1)g\right](t,z)
 =
 M_{t-}^f
 \left[B_{t-}^{f,g}(q-1)+gq\right](t,z),
\]
and the coefficient of $\lambda(\mathrm dt,\mathrm dz)$ is its
negative. Hence
\[
 \mathrm dV_t^{f,g}
 =
 \int_E K_{f,g}(t,z,j)\,(j-\lambda)(\mathrm dt,\mathrm dz).
\]
Since $B_0^{f,g}=0$, integration proves
\eqref{eq:mecke-pathwise-process}; its terminal value is
\eqref{eq:mecke-pathwise-terminal}.

For completeness, let $c=\norm{g}_\infty$ and $a=\lambda(C)$.
Since $f\geq0$ and $q\leq1$,
\[
 0<M_t^f(j)\leq1,\qquad
 |B_{t-}^{f,g}(j)|\leq c\bigl(j(C)+a\bigr).
\]
The integrand vanishes outside $C$ and satisfies
\begin{equation}\label{eq:mecke-k-bound}
 |K_{f,g}(t,z,j)|
 \leq c\bigl(j(C)+a+1\bigr)\one_C(t,z).
\end{equation}
Both integrals in \eqref{eq:mecke-pathwise-process} are therefore
absolutely convergent.

To identify the jump derivative, keep $\lambda$ fixed and
insert $(t,z)$ into the strict past of $j$. Atomlessness of
$\lambda$ implies that the values of $M^f$ and $B^{f,g}$ on
that strict past, evaluated at time $t$, are $M_{t-}^f$ and
$B_{t-}^{f,g}$. Insertion changes them to
$M_{t-}^f q(t,z)$ and $B_{t-}^{f,g}+g(t,z)$, respectively.
Their product therefore changes by
\[
 M_{t-}^f q(t,z)\bigl(B_{t-}^{f,g}+g(t,z)\bigr)
   -M_{t-}^f B_{t-}^{f,g}
 =K_{f,g}(t,z,j),
\]
which proves \eqref{eq:mecke-functional-derivative}.

We now evaluate the identity at $J$. Compensation gives
$\E J(C)=a$, so $J(C)<\infty$ almost surely. The processes
$M^f(J)$ and $B^{f,g}(J)$ are adapted and c\`adl\`ag; hence
their left limits, and therefore $K_{f,g}(\cdot,\cdot,J)$,
are predictable. By the compensated-integral isometry,
applied to $\one_C$,
\begin{equation}\label{eq:mecke-second-moment}
 \E\bigl[(J(C)-a)^2\bigr]=a,\qquad
 \E[J(C)^2]=a+a^2.
\end{equation}
These identities use only compensation, time simplicity, and
\eqref{eq:mecke-time-atomless}. They do not use the Poisson
distribution or the Mecke formula. From
\eqref{eq:mecke-k-bound}--\eqref{eq:mecke-second-moment},
\[
 \E\int_S |K_{f,g}(t,z,J)|^2\,\lambda(\mathrm dt,\mathrm dz)
 \leq c^2 a\,\E\bigl[(J(C)+a+1)^2\bigr]<\infty.
\]
The pathwise compensated integral in
\eqref{eq:mecke-pathwise-process} agrees with its stochastic
integral. It is therefore a square-integrable martingale.
The inverse characterization in Theorem~\ref{thm:closure}
gives \eqref{eq:mecke-closed-derivative}.
\end{proof}
The next result extends the terminal identity from exponential
functionals to arbitrary bounded measurable test functions. The predictable
representation property for the filtration is not required.
\begin{corollary}[Mecke remainder and Mecke formula]
\label{cor:mecke-remainder}
Under the stochastic assumptions of
Proposition~\ref{prop:mecke-pathwise}, let
$H:\mathcal N_\lambda\times S\to\R$ be bounded and measurable,
and suppose that $H(j,x)=0$ for $x\notin A$, where $A$ is
deterministic and $\lambda(A)<\infty$. Define
\begin{equation}\label{eq:mecke-remainder}
 R_H
 =
 \int_S H(J,x)\,J(\mathrm dx)
 -
 \int_S H(J+\delta_x,x)\,\lambda(\mathrm dx).
\end{equation}
Then there is a unique
$k_H\in {\mathcal L^2_{\Pp}(\lambda)} $ such that
\begin{equation}\label{eq:mecke-full-pathwise}
 \int_S H(J,x)\,J(\mathrm dx)
 =
 \int_S H(J+\delta_x,x)\,\lambda(\mathrm dx)
 +
 \int_S k_H(t,z)\,(J-\lambda)(\mathrm dt,\mathrm dz)
 \quad\text{almost surely}.
\end{equation}
In particular $\E[R_H\mid\mathcal F_0]=0$. If
$Y_t^H=\E[R_H\mid\mathcal F_t]$, with its c\`adl\`ag version,
then
\begin{equation}\label{eq:mecke-remainder-derivative}
 k_H=\nabla_J^{\Pp}Y^H,\qquad
 \E\int_S |k_H|^2\,\mathrm d\lambda=\E[R_H^2].
\end{equation}
Consequently \eqref{eq:mecke-classical}, with $J$ in place of
$\eta$, holds for every nonnegative measurable $H$.
\end{corollary}

\begin{proof}
Let
\[
 \mathcal R_\lambda
 =
 \left\{
  \int_S k(t,z)\,(J-\lambda)(\mathrm dt,\mathrm dz):
  k\in {\mathcal L^2_{\Pp}(\lambda)} 
 \right\}
 \subset L^2(\mathcal F_T).
\]
The isometry implies that $\mathcal R_\lambda$ is a closed
linear subspace. Every element of $\mathcal R_\lambda$ has
conditional expectation zero given $\mathcal F_0$.

Fix a measurable $A$ with $\lambda(A)<\infty$, and let
$\mathcal V_A$ be the vector space of bounded measurable
functions on $\mathcal N_\lambda\times A$ whose extension
by zero to $\mathcal N_\lambda\times S$ satisfies
$R_H\in\mathcal R_\lambda$. The remainder is in $L^2$ for every
such bounded function, since
\begin{equation}\label{eq:mecke-remainder-bound}
 |R_H|
 \leq \norm{H}_\infty\bigl(J(A)+\lambda(A)\bigr),
\end{equation}
and \eqref{eq:mecke-second-moment} applies to $A$.
Take a bounded nonnegative $f$ with ${\rm supp}(f)\subset S_n$, and a
bounded measurable $g$ vanishing outside $A$. For the tensor
function
\[
 H_{f,g}(j,x)=e^{-j(f)}g(x),
\]
insertion gives
$H_{f,g}(J+\delta_x,x)=e^{-J(f)}e^{-f(x)}g(x)$.
Thus
\[
 R_{H_{f,g}}
 =
 e^{-J(f)}\bigl(J(g)-\lambda(ge^{-f})\bigr)
 =
 \int_S K_{f,g}(t,z,J)\,(J-\lambda)(\mathrm dt,\mathrm dz),
\]
by Proposition~\ref{prop:mecke-pathwise}, applied on the
finite-measure set $S_n\cup A$. Hence $H_{f,g}\in\mathcal V_A$.

The finite linear span of these tensor functions is a unital
algebra on $\mathcal N_\lambda\times A$. It generates the product
$\sigma$-field. Indeed, the functions $g$ generate the
$\sigma$-field on $A$. For every measurable $D\subset S$,
the choice $f=\one_{D\cap S_n}$ gives
$e^{-j(D\cap S_n)}$. Since $j(S_n)<\infty$, this determines
$j(D\cap S_n)$, and
\[
 j(D)=\lim_{n\to\infty}j(D\cap S_n).
\]
Consequently these exponential functions generate the evaluation
$\sigma$-field of $\mathcal N_\lambda$. Conversely, each such
function is evaluation-measurable.

The space $\mathcal V_A$ is stable under uniformly bounded
pointwise limits. To see this, let $H_n\in\mathcal V_A$,
$\sup_n\norm{H_n}_\infty\leq b$, and $H_n\to H$ pointwise on
$\mathcal N_\lambda\times A$. On the event $\{J(A)<\infty\}$,
bounded convergence applies both to integration
against $J|_A$ and to integration against $\lambda|_A$, including
the inserted configurations in the second integral. Hence
$R_{H_n}\to R_H$ almost surely. Moreover,
\[
 |R_{H_n}-R_H|
 \leq 2b\bigl(J(A)+\lambda(A)\bigr),
\]
whose square is integrable. Dominated convergence therefore yields
$R_{H_n}\to R_H$ in $L^2$. Closedness of $\mathcal R_\lambda$
implies $R_H\in\mathcal R_\lambda$.
The functional monotone-class theorem now shows that
$\mathcal V_A$ contains every bounded measurable function on
$\mathcal N_\lambda\times A$. Thus $R_H$ has the representation
\eqref{eq:mecke-full-pathwise}. The integrand is unique in
${\mathcal L^2_{\Pp}(\lambda)} $ by the isometry.
Conditional expectation of that representation gives
\[
 Y_t^H
 =
 \int_{(0,t]\times E}k_H(s,z)\,
                          (J-\lambda)(\mathrm ds,\mathrm dz),
\]
including $Y_0^H=0$. Theorem~\ref{thm:closure} and the isometry
give \eqref{eq:mecke-remainder-derivative}. In particular,
taking expectations in \eqref{eq:mecke-full-pathwise} establishes
the Mecke formula for bounded test functions with finite-measure
support in the time--mark variable.

Finally, for arbitrary nonnegative measurable $H$, set
\[
 H_n(j,x)=\min\{H(j,x),n\}\one_{S_n}(x).
\]
Both integrands in \eqref{eq:mecke-classical} increase to their
respective limits. Applying the bounded case to $H_n$ and then
the monotone convergence theorem proves
\eqref{eq:mecke-classical}, with values in $[0,\infty]$.
\end{proof}

\begin{remark}
Identity~\eqref{eq:mecke-pathwise-terminal} holds for every
admissible configuration, whereas
\eqref{eq:mecke-full-pathwise} is an almost-sure identity for
each fixed $H$, with its integrand specified by the stochastic derivative. This property does not require a simultaneous
choice of versions for all measurable $H$. Neither proof assumes
the Poisson representation theorem or the Malliavin projection
identity.
\end{remark}

\providecommand{\bysame}{\leavevmode\hbox to3em{\hrulefill}\thinspace}
\providecommand{\MR}{\relax\ifhmode\unskip\space\fi MR }
\providecommand{\MRhref}[2]{%
  \href{http://www.ams.org/mathscinet-getitem?mr=#1}{#2}
}
\providecommand{\href}[2]{#2}

\end{document}